\documentclass[10pt]{amsart}

\usepackage{amssymb}
\usepackage[DIV10, headinclude, BCOR 10mm]{typearea}
\usepackage{enumerate} 
\usepackage[colorlinks,linkcolor=blue,citecolor=blue,urlcolor=blue,a4paper]{hyperref}
\usepackage{amsrefs} 
\hypersetup{pdfauthor={Wolfgang Lück, Roman Sauer, Christian Wegner},
  pdftitle ={L2-Torsion, the measure-theoretic determinant conjecture,
    and uniform measure equivalence}, }

\theoremstyle{plain} \newtheorem{theorem}{Theorem}[section]
\newtheorem*{theorem-non}{Theorem}
\newtheorem{proposition}[theorem]{Proposition}
\newtheorem{corollary}[theorem]{Corollary}
\newtheorem*{corollary-non}{Corollary}
\newtheorem{conjecture}[theorem]{Conjecture}
\newtheorem{convention}[theorem]{Convention}
\newtheorem*{conjecture-non}{Conjecture}
\newtheorem{lemma}[theorem]{Lemma} \newtheorem*{lemma-non}{Lemma}

\theoremstyle{definition} \newtheorem{definition}[theorem]{Definition}

\theoremstyle{remark} \newtheorem{remark}[theorem]{Remark}
\newtheorem{example}[theorem]{Example}
\newtheorem*{example-non}{Example}

\numberwithin{equation}{section}

\newcommand{\C}{\mathbb C} \newcommand{\N}{\mathbb N}
 \newcommand{\R}{\mathbb R}
\newcommand{\Z}{\mathbb Z}

\newcommand{\cala}{{\mathcal A}} \newcommand{\calb}{{\mathcal B}}
 
\newcommand{\caln}{{\mathcal N}} 
 \newcommand{\calr}{{\mathcal R}}

\newcommand{\bbC}{{\mathbb C}} \newcommand{\bbN}{{\mathbb N}}
\newcommand{\bbR}{{\mathbb R}} \newcommand{\bbZ}{{\mathbb Z}}

\DeclareMathOperator{\cell}{cell} \DeclareMathOperator{\colim}{colim}
\DeclareMathOperator{\cone}{cone} \DeclareMathOperator{\diag}{diag}
\DeclareMathOperator{\even}{even} \DeclareMathOperator{\im}{im}
\DeclareMathOperator{\odd}{odd} \DeclareMathOperator{\pr}{pr}
\DeclareMathOperator{\tr}{tr} \DeclareMathOperator{\id}{id}

\newcommand{\action}{\curvearrowright}
\newcommand{\abs}[1]{{\left\lvert #1\right\rvert}}
\newcommand{\fixabs}[1]{{\lvert #1\rvert}}
 \newcommand{\torsion}{\rho^{(2)}}

\renewcommand{\Im}{\operatorname{im}}
\newcommand{\ind}{\operatorname{ind}}
\newcommand{\map}{\operatorname{map}}
\newcommand{\norm}[1]{{\left\lVert #1\right\rVert}}
\newcommand{\determ}{{\textstyle\operatorname{det}}}
\newcommand{\fixnorm}[1]{{\lVert #1\rVert}}
\newcommand{\res}{\operatorname{res}}

\newcommand{\trace}{\operatorname{tr}}

\newcommand{\astvn}{\operatorname{\ast_{vN}}}

\title[$L^2$-Torsion and uniform orbit equivalence]{L$^2$-Torsion, the
  measure-theoretic determinant conjecture, and uniform measure
  equivalence}

\author{Wolfgang L\"uck} \author{Roman Sauer} \author{Christian
  Wegner}

\thanks{We gratefully acknowledge financial support by the
  Sonderforschungsbereich 478 \emph{Geometrische Strukturen in der
    Mathematik} and by the \emph{Max-Planck-Forschungspreis} and the
  \emph{Leibniz-Preis} of the first author.}

\address{Mathematisches Institut \\ Universit\"at M\"unster \\
  Einsteinstra{\ss}e 62 \\ M\"unster, D-48149 \\ Germany}
\email{lueck@uni-muenster.de} \email{sauerr@uni-muenster}
\email{c.wegner@uni-muenster.de}

\urladdr{http://wwwmath.uni-muenster.de/u/lueck/}
\urladdr{http://wwwmath.uni-muenster.de/u/sauerr/}
\urladdr{http://wwwmath.uni-muenster.de/reine/inst/lueck/homepages/christian\_wegner/wegner.html}
\begin{document}
\maketitle

\begin{abstract}
  We show an invariance result for the $L^2$-torsion of groups under
  uniform measure equivalence provided a measure-theoretic version of
  the determinant conjecture holds. The measure-theoretic determinant
  conjecture is discussed and, for instance, proved for Bernoulli
  actions of residually amenable groups.
\end{abstract}

\section{Introduction} Gaboriau~\cite{gaboriau-main} introduced
\emph{$L^2$-Betti numbers of measured equivalence relations} and
proved that two measure equivalent countable groups have proportional
$L^2$-Betti numbers.  This notion turned out to have many important
applications in recent years, most notably through the work of
Popa~\cite{popa-survey}.

In the present paper we study another well known $L^2$-invariant of a
discrete group $G$, the \emph{$L^2$-torsion} $\torsion(G)$, with
regard to measure equivalence.  The $L^2$-torsion
(Definition~\ref{def:l2 torsion}) is only defined if all the
$L^2$-Betti numbers of $G$ vanish and the \emph{determinant
  conjecture} (see Definition~\ref{def:determinant conjecture for
  groups}) -- an integral relative of the \emph{Connes' embedding
  problem} (see Remark~\ref{rem:connes embedding conjecture})-- holds 
for $G$. The determinant conjecture is intensively
studied~\cites{burgh,Sch, approx}, and there is no counterexample
known.  Notably, all sofic groups satisfy the determinant
conjecture~\cite{ES1}.

The notion of \emph{measure equivalence} was introduced by
Gromov~\cite{Gro}*{0.5.E} and, for the first time, gained prominence
in the work of Furman~\cite{Fur1}*{Definition~1.1}:

\begin{definition}\label{def:measure equivalence}
  Two countable groups $G$ and $H$ are called \emph{measure equivalent
    with index $c=I(G,H)>0$} if there exists a non-trivial standard
  measure space $(\Omega,\mu)$ on which $G\times H$ acts such that the
  restricted actions of $G=G\times\{1\}$ and $H=\{1\}\times H$ have
  measurable fundamental domains $X\subset\Omega$ and
  $Y\subset\Omega$, with $\mu(X)<\infty$, $\mu(Y)<\infty$, and
  $c=\mu(X)/\mu(Y)$. The space $(\Omega,\mu)$ is called a
  \emph{measure coupling} between $G$ and $H$ (of index $c$).
\end{definition}

Evidence for the following conjectural analog
(compare~\cite{Luc2}*{Question~7.35 on p.~313}) of the aforementioned
result by Gaboriau comes from computations and the formal similarities
between $\torsion(G)$ and the Euler characteristic of $G$.

\begin{conjecture}\label{con:ultimate conjecture}
  Let $G$ and $H$ be countable groups such that all the $L^2$-Betti
  numbers of $G$ and $H$ vanish. Assume that both $G$ and $H$ admit finite
  CW-models for their classifying spaces.  Then $G$ and $H$ satisfy the determinant 
conjecture. If $G$ and $H$ are measure
  equivalent with index $c=I(G,H)>0$, then
  \[
  \torsion(G)=c\cdot\torsion(H).
  \]
\end{conjecture}

If one attempts to solve the conjecture by eventually reducing it to
homotopy invariance, which is done in~\cite{gaboriau-main} for
$L^2$-Betti numbers (see~\cite{Sau} for another approach), then one
encounters at least two difficulties that do not appear for
$L^2$-Betti numbers: one is the need of a measure-theoretic version of
the determinant conjecture (Definition~\ref{def: measure theoretic
  determinant conjecture} below), the other is that the
Fuglede-Kadison determinant lacks a continuity property that is
obvious for the trace (discussed in Section~\ref{sec:caveat}). The latter 
difficulty disappears if we restrict to uniformly measure equivalent groups 
in Conjecture~\ref{con:ultimate conjecture}. 

The proof of the invariance of the 
$L^2$-torsion under \emph{uniform measure equivalence}, which is the main 
topic of this paper, is nevertheless much more involved than the one of the 
invariance of $L^2$-Betti numbers under uniform measure equivalence. 

\begin{definition}\label{def:uniform ME in intro}
  Two countable groups are \emph{uniformly measure equivalent of index
    $c=I(G,H)$} if there exists a measure coupling $(\Omega,\mu)$
  between $G$ and $H$ of index $c$ with a measurable $G$-fundamental
  domain $X$ and measurable $H$-fundamental domain $Y$ such that the
  following two conditions hold:
  \begin{enumerate}
  \item for every $g\in G$ there is a finite subset $H(g)\subset H$
    such that $gY\subset H(g)Y$ up to $\mu$-null sets, and
  \item for every $h\in H$ there is a finite subset $G(h)\subset G$
    such that $hX\subset G(h)X$ up to $\mu$-null sets.
  \end{enumerate}
\end{definition}

Uniform measure equivalence was introduced by Shalom and studied in
the context of quasi-isometry of amenable groups by Shalom and the
second author~\cite{shalom-amenable, sauer-amenable}.  Uniform measure
equivalence is much more restrictive and geometric than measure
equivalence. Here are important examples:
 
\begin{example}\label{exa:examples of UME}\hfill
  \begin{enumerate}
  \item Two uniform lattices in the same locally compact, second
    countable Hausdorff group are uniformly measure equivalent.
  \item Two finitely generated amenable groups are uniformly measure
    equivalent if and only if they are
    quasi-isometric~\citelist{\cite{shalom-amenable}*{Theorem~2.1.7}\cite{sauer-promotion}*{Lemma~2.25}}.
  \end{enumerate}
\end{example}

Next we introduce a measure-theoretic version of the determinant
conjecture. Some definitions are in order:

Let $\calr$ be a standard equivalence relation on a standard
probability space $(X,\mu)$ with an invariant measure $\mu$ in the
sense of~\cite{feldman+mooreI}*{Section~2}. We call an equivalence
relation with these properties just a \emph{measured equivalence
  relation}.  We call an action of a countable group $G$ on a standard
probability space $(X,\mu)$ $\emph{standard}$ if it is measurable and
$\mu$ is $G$-invariant.  Every measured equivalence relation on
$(X,\mu)$ is the orbit equivalence relation of some standard action of
a countable group on $(X,\mu)$~\cite{feldman+mooreI}*{Theorem~1}.

\begin{definition}\label{def:groupoid ring}
  The \emph{groupoid ring} of $\calr$ is, as an additive group,
  defined as
  \begin{multline*}
  \bbC\calr=\bigl\{f:\calr\rightarrow\bbC~\vert~\exists N\in\bbN\colon\text{$f^{-1}(\bbC^\times)\cap
    \{x\}\times X$ and $f^{-1}(\bbC^\times)\cap X\times\{x\}$ have}\\\text{ cardinality 
    at most $N$ a.e.}\bigr\}\subset L^\infty(\calr,\bbC).
  \end{multline*}
  endowed with the following multiplication, involution, and trace,
  respectively:
  \begin{enumerate}
  \item $(fg)(x,y)=\sum_{z\sim x}f(x,z)g(z,y)$,
  \item $f^\ast(x,y)=f(y,x)$,
  \item $\tr_{\caln(\calr)}(f)=\int_X f(x,x)d\mu(x)$.
  \end{enumerate}
  The \emph{integral groupoid ring} is $\bbZ\calr=\bbC\calr\cap
  L^\infty(\calr,\bbZ)$. The finite von Neumann algebra of
  $\calr$~\cite{feldman+mooreII}*{Section~2}, which contains
  $\bbC\calr$ as a weakly dense subset and to which
  $\tr_{\caln(\calr)}$ extends as a finite trace, is denoted by
  $\caln(\calr)$. If $\calr$ is the orbit equivalence relation of a
  standard action $G\action X$, we will also use the notation
  $\caln(G\action X)$ instead of $\caln(\calr)$.
\end{definition}

\begin{definition}\label{def: measure theoretic determinant conjecture}
  Let $\calr$ be a standard equivalence relation and $G$ be a countable
  group.  We say that $\calr$ satisfies the measure-theoretic
  determinant conjecture (abbreviated as \emph{MDC}) if the
  (generalized) Fuglede-Kadison determinant of every matrix $A\in
  M(m\times n, \bbZ\calr)$ satisfies
  \[
  \determ_{\caln(\calr)}(A)\ge 1.
  \]
\end{definition}
A standard action of $G$ satisfies MDC if its orbit equivalence
relation satisfies MDC. The group $G$ satisfies MDC if every
essentially free, standard action of $G$ satisfies MDC.

\begin{conjecture}[Measure Theoretic Determinant Conjecture]
  \label{con: Measure Theoretic Determinant Conjecture}
  Every standard equivalence relation satisfies the measure-theoretic
  determinant conjecture.
\end{conjecture}

The following theorem, actually a stronger version thereof, is proved
in Section~\ref{sec:MDC}.

\begin{theorem}\label{thm:intro thm about MDC} 
  Let $G$ be a countable group.
  \begin{enumerate}
  \item Countable amenable groups satisfy MDC.
  \item If $G$ satisfies MDC, then every subgroup of $G$ satisfies
    MDC.
  \item If $G$ satisfies MDC, then every amenable extension of $G$
    satisfies MDC.
  \item Let $1\rightarrow K\rightarrow G\rightarrow Q\rightarrow 1$ be
    a group extension such that $K$ is finite. If $G$ satisfies MDC,
    then $Q$ satisfies MDC.
  \item Assume that $G = \colim_{i \in I} G_i$ where $(G_i)_{i\in I}$
    is a directed system of countable groups (whose structure maps are not necessarily injective). If every $G_i$ satisfies
    MDC, then $G$ satisfies MDC.
  \item Bernoulli actions of countable residually amenable groups
    satisfy MDC.
  \end{enumerate}
\end{theorem}

\begin{theorem}\label{thm:MDC implies det class for ME groups}
  Let $G$ and $H$ be measure equivalent groups. If $G$ satisfies MDC,
  then $H$ satisfies the determinant conjecture.
\end{theorem}

The preceding theorem is proved in Section~\ref{sec:uniform measure
  equivalence invariance}. In the same section we also prove the
following theorem as a first step towards Conjecture~\ref{con:ultimate
  conjecture}. The conclusion about the vanishing of the $L^2$-Betti
numbers of $H$ below is of course due to the corresponding theorem of
Gaboriau~\cite{gaboriau-main}, which we built in for a clean
formulation.

\begin{theorem}\label{thm:invariance under uniform ME}
  Let $G$ and $H$ be groups that admit finite CW-models for their
  classifying spaces. Assume that all the $L^2$-Betti numbers of $G$
  vanish and $G$ satisfies MDC.  If $G$ and $H$ are uniformly measure
  equivalent with index $c=I(G,H)>0$, then all the $L^2$-Betti numbers
  of $H$ vanish, $H$ satisfies the determinant conjecture, and
  \[
  \torsion(G)=c\cdot\torsion(H).
  \]
\end{theorem}

Example~\ref{exa:examples of UME} yields the following corollary:

\begin{corollary}\label{cor:qi invariance}
  Let $G$ and $H$ be amenable groups that admit finite CW-models for
  their classifying spaces. If $G$ and $H$ are quasi-isometric, then
  \[ \torsion(G)=0\iff\torsion(H)=0.\]
\end{corollary}

We emphasize that there is the conjecture that the $L^2$-torsion
vanishes for all infinite amenable groups.  In~\cite{wegner} it is shown that
$\torsion(G)=0$ if $G$ contains an infinite elementary amenable
normal subgroup and has a finite model of its classifying space. 

\section{Background in $L^2$-invariants and orbit equivalence
  relations}\label{sec:background}

\subsection{Uniform orbit equivalence}\label{sub:bounded OE}

The notion of \emph{orbit equivalence} has its roots in the pioneering
work of Dye~\cites{dye1,dye2}. We recall it here:

\begin{definition}\label{def:orbit equivalence}
  Two standard actions $G \curvearrowright (X,\mu_X)$ and $H
  \curvearrowright (Y,\mu_Y)$ are called \emph{weakly orbit equivalent
    with index $c=I(G,H)>0$ } if there are measurable subsets $A
  \subset X$, $B \subset Y$ and a measurable isomorphism $f: A \to B$
  such that
  \begin{enumerate}
  \item $G\cdot A=X$ and $H\cdot B=Y$ up to null sets,
  \item $\frac{1}{\mu_X(A)}f_\ast(\mu_X\vert_A)=
    \frac{1}{\mu_Y(B)}\mu_Y\vert_B$,
  \item $f(G \cdot x \cap A) = H \cdot f(x) \cap B$ for a.e.~$x\in A$,
    and
  \item $c=\mu_X(A)/\mu_Y(B)$.
  \end{enumerate}
  If $\mu_X(A)=\mu_Y(B)=1$, then we call the actions \emph{orbit
    equivalent}. The map $f$ is called a \emph{weak orbit equivalence}
  or \emph{orbit equivalence}, respectively.
\end{definition}

The following modification of orbit equivalence was introduced
in~\cite{sauer-promotion} (under the name \emph{bounded orbit
  equivalence}).

\begin{definition}
  Two standard actions $G \curvearrowright (X,\mu_X)$ and $H
  \curvearrowright (Y,\mu_Y)$ are called \emph{uniformly weakly orbit
    equivalent with index $c=I(G,H)$} if there exists a weak orbit
  equivalence $f:A\rightarrow B$ of index $c$ as in
  Definition~\ref{def:orbit equivalence} that satisfies the following
  additional properties:
  \begin{enumerate}
  \item There exist finite subsets $F_A\subset G$ and $F_B\subset H$
    such that $F_A \cdot A=X$ and $F_B \cdot B=Y$ up to null sets,
  \item For every $g\in G$ there is a finite subset $F(g)\subset H$
    such that $f(gx) \in F(g) \cdot f(x)$ for a.e.~$x \in A \cap
    g^{-1} \cdot A$.
  \item For every $h\in H$ there is a finite subset $F(h)\subset G$
    such that $f^{-1}(hy) \in F(h) \cdot f^{-1}(y)$ for a.e.~$y \in B
    \cap h^{-1} \cdot B$.
  \end{enumerate}
\end{definition}

The following theorem is proved in~\cite{Fur2}*{Theorem~3.3} (see
also~\cite{sauer-promotion}*{Theorem~2.33} for the uniform version).

\begin{theorem}\label{thm:ME implies OE}
  Two countable groups $G$ and $H$ are (uniformly) measure equivalent
  with respect to a measure coupling of index $c>0$ if and only if
  there exist essentially free, standard actions $G$ and $H$ that are
  (uniformly) weakly orbit equivalent of index $c$.
\end{theorem}

\subsection{Algebras associated to groups and equivalence
  relations}\label{sub: algebras associated to an OE relation}

\begin{definition}\label{def-L-ring}
  Let $G\action (X,\mu)$ be a standard action.  The \emph{crossed
    product ring} $L^\infty(X,\C)*G$ is the free
  $L^\infty(X,\C)$-module with $G$ as $L^\infty(X,\C)$-basis. Here
  $L^\infty(X,\C)$ denotes the ring of equivalence classes of
  essentially bounded measurable functions $X \to \C$. The
  multiplication is given by
  \[
  \big(\sum_{g \in G} r_g \cdot g\big) \cdot \big(\sum_{g \in G} s_g
  \cdot g\big) = \sum_{g \in G} \big(\sum_{\stackrel{\scriptstyle g_1,
      g_2 \in G}{g_1 g_2 = g}} r_{g_1} \cdot (s_{g_2} \circ
  m_{{g_1}^{-1}})\big) \cdot g
  \]
  with $m_g\colon X \to X, x \mapsto gx$.  The crossed product ring
  $L^\infty(X,\Z)*G$ is defined analogously using the ring of
  equivalence classes of essentially bounded measurable functions $X
  \to \Z$.
\end{definition}

\begin{remark}[compare~\cite{sauer-promotion}*{Sections~1.2 and~1.3}]\label{orbit-equivalence-relation}
  Let $G\action (X,\mu)$ be an essentially free, standard action and
  $\calr$ be its orbit equivalence relation.  The rings
  $L^\infty(X,\Z)*G$ and $L^\infty(X,\C)*G$ embed as subrings into
  $\bbZ\calr$ or $\bbC\calr$, respectively. Let
  $i_g\colon X\rightarrow\calr$ be the map $i_g(x)=(x,g^{-1}x)$. We obtain a
  (multiplicative) isomorphism
  \begin{equation*}
    L^\infty(X,\C)*G\cong
    \big\{ f \in\bbC\calr \, \big| \, f \circ i_g = 0 \in L^\infty(X,\C) 
    \text{ for all but finitely many } g \in G \big\}\subset\bbC\calr
  \end{equation*}
  given by $\sum_{g \in G} r_g \cdot g \mapsto f$ with $f(x,g^{-1}x) =
  r_g(x)$.  There is an analogous statement for integral (instead of
  complex) coefficients.  In particular, we obtain a trace-preserving
  inclusion $\bbC G\hookrightarrow\bbC\calr$; this inclusion extends
  to the von Neumann algebras $\caln(G)$ of $G$ and $\caln(\calr)$ of
  $\calr$.
\end{remark}

The following lemmas about the crossed product of a standard action
$G\action X$ are rather easy to verify.

\begin{lemma}\label{lem:flat}
  $L^\infty(X,\Z)*G$ is flat over $\Z G$.
\end{lemma}

\begin{proof}
  This follows from the fact that $L^\infty(X,\bbZ)$ is torsionfree
  and the isomorphism
  \[
  L^\infty(X,\Z)*G \otimes_{\Z G} M \cong L^\infty(X,\bbZ) \otimes_\Z
  M.\qedhere
  \]
\end{proof}

Recall that that an idempotent $p$ in a ring $R$ is called \emph{full}
if the additive group generated by elements of the form $rpr'$,
$r,r'\in R$, is $R$. If $p\in R$ is full, then the rings $pRp$ and $R$
are Morita equivalent. This implies \emph{e.g.} that $P$ is a finitely
generated projective $R$-module if and only if $pP$ is a finitely
generated projective $pRp$-module~\cite{lam}*{(18.30B) on
  p.~490}. Note also that if $R\subset S$ is a unital subring and
$p\in R$ is a full idempotent in $R$, then $p$ is also a full
idempotent in $S$.

\begin{remark}\label{rem:restricted groupoid ring}
  Let $\calr$ be the orbit equivalence relation of a standard action
  of $G$ on $(X,\mu)$. Let $A\subset X$ be a subset. We denote the
  restricted equivalence relation by $\calr\vert_A=\calr\cap A\times
  A$. One has
  \[
  \bbZ\calr\vert_A=\chi_A\bbZ\calr\chi_A,
  \]
  where $\chi_A$ is the characteristic function of $A$. Similarly,
  $\caln(\calr\vert_A)=\chi_A\caln(\calr)\chi_A$.
\end{remark}

\begin{lemma}[\cite{sauer-promotion}*{Lemma~4.21}]\label{lem:full idempotent}
  Let $A\subset X$ be a measurable subset such that finitely many
  $G$-translates of $A$ cover $X$ up to null sets. Then the
  characteristic function $\chi_A$ is a full idempotent in
  $\chi_AL^\infty(X,\bbZ)\ast G\chi_A$.
\end{lemma}

\begin{lemma}[\cite{sauer-promotion}*{Lemma~4.23}]\label{lem:uniform OE induces ring iso}
  Let $f:A \to B$ be an orbit equivalence between essentially free,
  standard actions of $G$ on $(X,\mu_X)$ and $H$ on $(Y,\mu_Y)$. Let
  $\calr_1$ and $\calr_2$ be the orbit equivalence relations of
  $G\action X$ and $H\action Y$, respectively.
  \begin{enumerate}
  \item The isomorphism $\calr_1\vert_A\xrightarrow{f\times
      f}\calr_2\vert_B$ induces the trace-preserving
    $\ast$-isomorphism
    \[f^\ast: \bbZ\calr_2\vert_B\rightarrow\bbZ\calr_1\vert_A,~
    \phi\mapsto \phi\circ (f\times f).\] Thus it extends to an
    isomorphism
    $\caln(\calr_2\vert_B)\rightarrow\caln(\calr_1\vert_A)$.
  \item If $f$ is uniform, the isomorphism $f^\ast$ restricts to an
    isomorphism
    \[\chi_B L^\infty(Y,\Z)*H\chi_B\to\chi_A L^\infty(X,\Z)*G \chi_A\]
    of the embedded subrings.
  \end{enumerate}
\end{lemma}

\begin{convention}\label{con:notation for GNS construction}
  Let $G\action X$ be a standard action, $A\subset X$ be a measurable
  subset, and $\calr$ be the orbit equivalence relation of the action. We
  introduce some equivalent notations: Instead of $\caln(\calr)$ and
  $\caln(\calr\vert A)$ we also write $\caln(G\action X)$ and
  $\caln(G\action X\vert_A)$, respectively.  For the GNS-construction 
  $l^2(\caln(\calr\vert_A))$ of $\caln(\calr\vert_A)$ we write
  $L^2(\calr\vert_A)$ or $L^2(G\action X\vert_A)$.
\end{convention}

\subsection{The Fuglede-Kadison determinant and
  $L^2$-torsion}\label{sub:determinant and torsion}

In the sequel, let $\cala$ be a finite von Neumann algebra with trace
$\tr_\cala:\cala\to\bbC$. Our main examples are the group von Neumann
algebra $\caln(G)$ of a group $G$ and the von Neumann algebra
$\caln(\calr)$ of a standard equivalence relation. A Hilbert space
with a (left) $\cala$-action that embeds isometrically and
equivariantly as a closed subspace into a Hilbert sum of copies of
$l^2(\cala)$ is a called a \emph{Hilbert $\cala$-module}. A bounded
$\cala$-equivariant operator between Hilbert $\cala$-modules is called
a \emph{morphism of Hilbert $\cala$-modules}.  The trace $\tr_\cala$
extends in a natural way to positive morphisms of Hilbert
$\cala$-modules. Further, every Hilbert $\cala$-module $H$ has a
real-valued \emph{dimension} $\dim_\cala(H)\in [0,\infty]$. This
dimension satisfies, for example, $\dim_\cala(l^2(\cala)
p)=\trace_\cala(p)$ for any projection $p\in\cala$.  We refer
to~\cite{Luc2}*{Chapter~6} for more information. The following convention is
adopted:

\begin{convention}
  The $n\times n$-matrices $M(n\times n,\cala)$ over $\cala$ are again
  a von Neumann algebra. We equip $M(n\times n,\cala)$ with the
  unnormalized trace $\trace_{M(n\times
    n,\cala)}(A)=\sum_{i=1}^n\trace_\cala(A_{ii})$. If the context is
  clear, we just write $\trace_\cala$ instead of $\trace_{M(n\times
    n,\cala)}$.  There is a one-to-one correspondence between $A\in
  M(n\times n,\cala)$ and (left)-$\cala$-equivariant bounded operators
  $l^2(\cala)^n\rightarrow l^2(\cala)^n$ via right matrix
  multiplication. If we want to stress the point of view of $A$ as a
  bounded operator, we also use the notation $r_A$ for the right
  multiplication on $l^2(\cala)^n$.
\end{convention}

We remind the reader of the definition of the spectral density
function and the Fuglede-Kadison determinant: Let $f\colon U \to V$ be
a morphism of Hilbert $\cala$-modules of finite dimension. Denote by
$\big\{E_\lambda^{f^*f}\colon U \to U \mid \lambda \in \R\big\}$ the
family of spectral projections of the positive endomorphism $f^*f$.
The spectral projections have the properties $\|f(u)\| \leq \lambda
\cdot \|u\|$ for $u \in \im(E_{\lambda^2}^{f^*f})$ and $\|f(u)\| >
\lambda \cdot \|u\|$ for $0 \neq u \in \ker(E_{\lambda^2}^{f^*f})$.
The \emph{spectral density function} of $f$ is defined as
\[
F(f)\colon \R \to [0,\infty),~\lambda \mapsto
\tr_{\cala}\big(E_{\lambda^2}^{f^*f}\big).
\]
The spectral density function $F(f)$ is monotonous and
right-continuous.

\begin{definition}[\cite{Luc2}*{Definition~3.11 on p.~127}]
  Let $f\colon U \to V$ be a morphism of Hilbert $\cala$-modules of
  finite dimensions. The \emph{Fuglede-Kadison determinant} or just
  \emph{determinant} of $f$ is defined as
  \[
  \determ_{\cala}(f) = \begin{cases}
    \exp\big(\int_{0^+}^\infty\ln(\lambda) \, dF(f)(\lambda)\big)&\text{ if $\int_{0^+}^\infty \ln(\lambda)dF(f)(\lambda) > -\infty$},\\
    0&\text{ otherwise.}
  \end{cases}
  \]
  Here the integral is understood to be the Lebesgue-Stieltjes
  integral with respect to $F(f)$.
\end{definition}

The main properties of this determinant are described
in~\cite{Luc2}*{Theorem~3.14}.

\begin{remark}[Induction]\label{rem:inclusion of vN algebras}
  Let $\cala\hookrightarrow\calb$ be a trace-preserving inclusion of
  finite von Neumann algebras. Let $f:l^2(\cala)^m\rightarrow
  l^2(\cala)^n$ be a morphism of Hilbert $\cala$-modules. Then $f$ is
  given by right multiplication $r_A$ with an $m\times n$-matrix over
  $\cala$. The morphism of Hilbert $\calb$-modules
  $l^2(\calb)^m\rightarrow l^2(\calb)^n$ defined by right
  multiplication with the same matrix is denoted by
  $\ind_\cala^\calb(f)$. It is obvious that
  \[\tr_\calb\bigl(P(\ind_\cala^\calb(f))\bigr)=\tr_\cala\bigl(P(f)\bigr)\]
  for any complex polynomial $P$ and $m=n$. Normality of the trace yields that
  $F(\ind_\cala^\calb(f))=F(f)$, hence
  \begin{equation}\label{eq: det under induction}
    \determ_\calb\bigl(\ind_\cala^\calb(f)\bigr)=\determ_\cala(f). 
  \end{equation}
\end{remark}

\begin{remark}[Restriction]\label{rem:restriction of von Neumann algebras}
  Let $p\in\cala$ be a projection.  Then $p\cala p$ is a finite von
  Neumann algebra with normalized trace $\tr_{p\cala
    p}=\frac{1}{\tr_\cala(p)}\tr_\cala$~\cite{dixmier}*{Prop.~1 on
    p.~17}.  Let $f:U\rightarrow V$ be a morphism of finitely
  generated Hilbert $\cala$-modules.  The morphism $f^\ast f$
  decomposes as an orthogonal sum
  \[
  pU\oplus (1-p)U\xrightarrow{f^\ast f\vert_{pU}\oplus f^\ast
    f\vert_{(1-p)U}}pU\oplus (1-p)U.
  \]
  Since spectral calculus respects orthogonal sums, we obtain that
  $\Im (E^{f^\ast f\vert_{pU}}_{\lambda^2})=p\Im (E^{f^\ast
    f}_{\lambda^2})$.  Viewing $f\vert_{pU}$ as a morphism of Hilbert
  $p\cala p$-modules, we obtain that
  $F(f)=\tr_{\cala}(p)F(f\vert_{pU})$ provided $p$ is
  full~\cite{jones}*{Proposition~2.2.6 vii) on p.~26}.  From this we
  conclude that, if $\det_\cala(f)>0$, then
  \begin{equation}\label{eq:det under restriction}
    \ln\determ_\cala(f)= \tr_{\cala}(p)\cdot\ln\determ_{p\cala p}\bigl(f\vert_{pU}\bigr). 
  \end{equation}
\end{remark}

\begin{definition}
  Let $C_*$ a Hilbert $\cala$-chain complex. Suppose that
  \begin{enumerate}
  \item $C_*$ is dim-finite, i.e. $\dim_{\cala}(C_n) < \infty$ for all
    $n \in \Z$ and there exists $N \in \N$ such that $C_n = 0$ if
    $n<0$ or $n>N$,
  \item $b_n^{(2)}(C_*) = 0$ for all $n \in \Z$,
  \item $\det_{\cala}(c_n) > 0$ for all $n \in \Z$.
  \end{enumerate}
  We define its \emph{$L^2$-torsion} by
  \[
  \torsion(C_*)= -\sum_{n \in \Z} (-1)^n
  \ln\big(\determ_{\cala}(c_n)\big) \in \R.
  \]
\end{definition}
The following conjecture is true for all sofic groups~\cite{ES1}; no
example of a group that is not sofic is known.

\begin{definition}[\cite{Luc2}*{Conjecture~13.2 on p.~454}]\label{def:determinant conjecture for groups}
  We say that the group $G$ satisfies the \emph{determinant
    conjecture} or is of \emph{determinant class} if the
  Fuglede-Kadison determinant of every matrix $A\in M(m\times n, \bbZ
  G)$ satisfies
  \[
  \determ_{\caln(G)}(A)\ge 1.
  \]
\end{definition}

Let $X$ be a finite CW-complex with vanishing $L^2$-Betti numbers such
that $G=\pi_1(X)$ satisfies the determinant conjecture.  We define the
\emph{$L^2$-torsion} of $X$ as
\[\rho^{(2)}\bigl({\widetilde X}\bigr)=
\rho^{(2)}\bigl(l^2(G)\otimes_{\bbZ G}C^{\cell}_\ast({\widetilde
  X})\bigr).\] Since the differentials in the cellular chain complex
$C^{\cell}_\ast(\widetilde{X})$ are matrices over $\bbZ G$ with
respect to cellular bases and thus have positive determinant, this
definition is justified.  If $G$ is of determinant class, then this
definition only depends on the homotopy type of
$X$~\cite{Luc2}*{Lemma~13.6 on p.~456}.
 
\begin{definition}\label{def:l2 torsion}
  Let $G$ be a group that admits a finite CW-model $X$ of its
  classifying space $BG$.  Suppose that the group $G$ satisfies the
  determinant conjecture and that all its $L^2$-Betti numbers
  vanish. Then we define the \emph{$L^2$-torsion} of $G$ as
  $\rho^{(2)}(G) = \rho^{(2)}(\widetilde{X})$.
\end{definition}

\section{The measure-theoretic determinant conjecture}\label{sec:MDC}

The goal of this section is to prove Theorem~\ref{thm:intro thm about
  MDC}; we actually prove the following slightly stronger formulation:

\begin{theorem}\label{thm:precise MDC result} 
  Let $G$ be a countable group and $H\subset G$ a subgroup.
  \begin{enumerate}
  \item The trivial group satisfies MDC.
  \item If $G$ satisfies MDC, then $H$ satisfies MDC.
  \item Let $H\subset G$ be a normal subgroup such that the quotient
    $G/H$ is amenable.  Let $G\action X$ be a standard action.  If
    $H\action X$ satisfies MDC, then $G\action X$ satisfies MDC.
  \item Let $1\rightarrow K\rightarrow G\rightarrow Q\rightarrow 1$ be
    a group extension such that $K$ is finite. If $G$ satisfies MDC,
    then $Q$ satisfies MDC.
  \item Assume that $G = \colim_{i \in I} G_i$ where $(G_i)_{i\in I}$
    is a directed system  (with not necessarily injective structure maps) 
    of countable groups. If every $G_i$ satisfies
    MDC, then $G$ satisfies MDC.
  \item Let $(X,\mu)$ be a standard probability space. Suppose that
    $G$ is a directed limit of countable groups $G=\lim_{i\in I}G_i$
    such that the shift action of $G_i$ on the product space
    $(X,\mu)^{G_i}$ satisfies MDC. Then the shift action of $G$ on
    $(X,\mu)^G$ satisfies MDC.
  \end{enumerate}
\end{theorem}

Actions as in (6) are called \emph{Bernoulli actions.} The reason that we restrict 
to Bernoulli actions in (6) is that we do not know how to approximate an arbitrary 
$G$-action by actions of the groups $G_i$. 
The remainder of this section is devoted to the proof of the theorem above.

\subsection{The approximation lemma}

The proof of the next lemma is essentially the same as in the special
case of group von Neumann algebras. Such proofs are given
in~\citelist{\cite{Sch}*{Section~6}\cite{Luc2}*{Theorem~13.19 on
    p.~461}\cite{lueck-gafa}*{Lemma~2.5}}.

\begin{lemma}\label{lem:Approximation Lemma}
  Let $\cala$, $\cala_i$ $(i \in I)$ be finite von Neumann algebras
  with $I$ a directed set. Let $A \in M(d \times d';\cala)$ and $A_i
  \in M(d_i \times d'_i;\cala_i)$ be matrices with the following
  properties where $\Delta$, $\Delta_i$ are defined as $\Delta := A
  A^* \in M(d \times d;\cala)$, $\Delta_i := A_i A_i^* \in M(d_i
  \times d_i;\cala_i)$:
  \begin{enumerate}
  \item $\det_{\cala_i}(A_i) \geq 1$,
  \item there exists a constant $K>0$ with $\norm{r_\Delta}\le K$ and
    $\norm{r_{\Delta_i}}\le K$,
  \item $\lim_{i \in I} \frac{\tr_{\cala_i}(\Delta_i^m)}{d_i} =
    \frac{\tr_{\cala}(\Delta^m)}{d}$ for all $m \in \N$.
  \end{enumerate}
  Then $\lim_{i \in I} \dim_{\cala_i}(\ker A_i) = \dim_{\cala}(\ker
  A)$ and $\det_{\cala}(A) \ge 1$.
\end{lemma}

\begin{remark}\label{rem:connes embedding conjecture}
	In Connes' pioneering paper~\cite{connes} the question was raised whether every 
	finite von Neumann algebra embeds into an ultraproduct of the 
	hyperfinite II$_1$-factor (nowadays referred to as the \emph{Connes embedding problem}). 
	If the Connes embedding problem has a positive answer for the finite von Neumann 
	algebra $\cala$, then for every self-adjoint $\Delta\in\cala$ there is a sequence 
	of matrices $\Delta_i\in M(d_i\times d_i,\bbC)$ that satisfies (3) of 
	Lemma~\ref{lem:Approximation Lemma}. If the $\Delta_i$ have only integral entries, then 
	$\determ_\cala (\Delta)\ge 1$. In that 
	regard the determinant conjecture is an integral relative of the 
	Connes embedding problem. 
\end{remark}

The following lemma is often helpful for verifying the second
condition in Lemma~\ref{lem:Approximation Lemma}. 
We omit its proof which is essentially the same as
in~\cite{Luc2}*{Proof of Lemma~13.33 on p.~466}.
 
\begin{lemma}\label{lem: uniform bound on the norm}
  Let $G \curvearrowright X$ be a standard action and $A \in M(d
  \times d';L^\infty(X,\C)*G)$. For an element $f = \sum_{g \in G} f_g
  \cdot g \in L^\infty(X,\C)*G$, let $\|f\|_\infty=\sum_{g \in G}
  \|f_g\|_{\infty}$.  Then:
  \[
  \|r_A\| \leq d \cdot d' \cdot \max_{k,l} \|A_{k,l}\|_\infty.
  \]
\end{lemma}

\subsection{Some reductions used in the proof}
\label{sub:some_reductions_used_in_the_proof}

\begin{lemma}\label{lem:can restrict to crossed product ring}
  Let $G\action (X,\mu)$ be an essentially free, standard action of a
  countable group $G$.  Let $\calr$ be its orbit equivalence relation
  on $X$.  Assume that $\determ_{\caln(\calr)}(A)\ge 1$ for every
  matrix $A\in M(d\times d';L^\infty(X,\bbZ)\ast G)$. Then $\calr$
  satisfies MDC.
\end{lemma}

\begin{proof}
  Let $A\in M(d\times d';\bbZ\calr)$. Choose an enumeration
  $G=\{g_1,g_2,\dots\}$.  We define an increasing sequence
  $(X_n)_{n\ge 1}$ of measurable subsets of $X$ by
  \[
  X_n=\Bigl\{x\in X\vert A_{ij}(x,g_mx)=0\text{ for $m>n$ and all
    $1\le i\le d$, $1\le j\le d'$}\Bigr\}.
  \]
  Obviously, $\mu(X_n)\rightarrow 1$. Set $A_n=\chi_{X_n}A$. Then
  $A_n\in M(d\times d';L^\infty(X,\bbZ)\ast G)$ and
  $\norm{r_{A_nA_n^*}}=\norm{r_{\chi_{X_n}AA^\ast\chi_{X_n}}}\le\norm{r_{AA^\ast}}$.
  By the continuity of the trace and the trace property, we obtain
  that
  \[
  \tr_{\caln(\calr)}(AA^\ast)=
  \lim_{n\rightarrow\infty}\tr_{\caln(\calr)}(\chi_{X_n}AA^\ast)=
  \lim_{n\rightarrow\infty} \tr_{\caln(\calr)}(A_nA_n^\ast).
  \]
  The assertion now follows from Lemma~\ref{lem:Approximation Lemma}.
\end{proof}

\begin{lemma}\label{lem:positive is enough}
  If $\det_{\caln(G\action X)}(A)\ge 1$ for every $n\ge 1$ and every
  positive matrix $A \in M(n \times n;L^\infty(X,\Z)\ast G)$, then
  $\det_{\caln(G\action X)}(B)\ge 1$ holds for all $m,n\ge 1$ and
  every matrix $B\in M(m \times n;L^\infty(X,\Z)\ast G)$.
\end{lemma}
\begin{proof}
  This directly follows from the identity
  \[\determ_{{\mathcal N}(G \curvearrowright X)}(B) =
  \determ_{{\mathcal N}(G \curvearrowright X)}(B B^*)^{1/2}.\qedhere\]
\end{proof}

At certain stages in the proof of Theorem~\ref{thm:precise MDC result}
it is convenient to allow for the flexibility of non-free actions on a
probability space. Let $G\action (X,\mu)$ be a, not necessarily free,
standard action. The crossed product ring $L^\infty(X)\ast G$ with its
trace can be completed to a von Neumann algebra $L^\infty(X)\astvn
G$. This von Neumann algebra is the von Neumann algebra associated to
the translation groupoid of the action~\cite{take}*{XIII~\S 2}.  If
the action is essentially free, then $L^\infty(X)\ast G=\caln(\calr)$
where $\calr\subset X\times X$ is the orbit equivalence relation of
$G\action X$. On the other extreme, if $X$ is just a point, then we
have $\caln(\calr)=\bbC$ and $L^\infty(X)\astvn G=\caln(G)$.

\begin{lemma}\label{lem:reduction for non-free actions}
  Assume that $G$ satisfies MDC. Let $G\action X$ be a (not
  necessarily free) standard action. Then
  \[
  \determ_{L^\infty(X)\astvn G}(A)\ge 1
  \]
  for every $A\in M(m\times n,L^\infty(X;\bbZ)\ast G)$.
\end{lemma}

\begin{proof}
  Let $G\action Y$ be an essentially free, standard action; take, for
  example, $Y=[0,1]^G$ with its shift action. Then the diagonal action
  of $G$ on the product probability space $X\times Y$ is essentially
  free. The projection $\pr:X\times Y\rightarrow X$ induces a
  trace-preserving $\ast$-homomorphism
  \[
  \pr^\ast\colon L^\infty(X)\ast G\rightarrow L^\infty(X\times Y)\ast
  G,\quad\sum f_g\cdot g\mapsto\sum (f_g\circ \pr)\cdot g,
  \]
  which extends to the von Neumann algebras. Since $G\action X\times
  Y$ satisfies MDC by hypothesis, the assertion follows (see
  Remark~\ref{rem:inclusion of vN algebras}).
\end{proof}

The remainder of this section is devoted to the proof of
Theorem~\ref{thm:precise MDC result}. Because of Lemma~\ref{lem:can
  restrict to crossed product ring} it suffices in each case to show
MDC only for matrices in the crossed product ring.

\subsection{The trivial group and transition to subgroups}
\label{sub:trivial and subgroup}

\begin{proof}[Proof of Theorem~\ref{thm:precise MDC result} (1)]
  Let $(X,\mu)$ be a standard probability space.  Let $A \in M(m\times
  n,L^{\infty}(X,\bbZ))$. We have to show that
  \[\determ_{L^{\infty}(X)}(A) \ge 1.\]
  By~\cite{Luc1}*{Lemma~4.1} there is a unitary matrix $U \in
  M(m\times m, L^{\infty}(X))$ such that $U^{-1}AA^*U$ is a diagonal
  matrix whose diagonal entries $f_1, f_2,\dots, f_m\in L^\infty(X)$
  are positive functions.  We conclude from~\cite{Luc2}*{Theorem 3.14
    (1) and (4) and Lemma 3.15 (3), (4) and (7) on p.~128/129} that
  \[
  \determ_{L^\infty(X)}(A)^2 = \prod_{i=1}^m
  \determ_{L^{\infty}(X)}(f_i).
  \]
  According to~\cite{Luc2}*{Example~3.13 on p.~128} we have
  \[
  \determ_{L^{\infty}(X)}(f_i)= \exp\Bigl(\int_X \ln(f_i(x)) \cdot
  \chi_{\{x \in X \mid f_i(x) > 0\}}d\mu(x)\Bigr).
  \]
  Combining the aforementioned equalities yields
  \begin{equation}\label{eq:product log}
    \determ_{L^{\infty}(X)}(A)^2= 
    \exp\Bigl( \int_X \ln\bigl(\prod_{\substack{i=1,2, \ldots ,m\\f_i(x) > 0}} f_i(x)\bigr) d\mu(x) \Bigr).
  \end{equation}
  Fix $x\in X$. Then $A(x)A(x)^*$ is a matrix in $M(m\times m,\bbZ)$.
  Let $p(t)$ be its characteristic polynomial. It can be written as
  $p(t) = t^a \cdot q(t)$ for a polynomial $q(t)$ with integer
  coefficients and $q(0) \not= 0$. Then $q(0)$ is the product of the
  positive eigenvalues of $A(x)A(x)^\ast$, i.e.
  \[\prod_{\substack{i=1,2, \ldots ,m\\f_i(x) > 0}} f_i(x)=q(0).\]
  Now the assertion follows from $q(0)\ge 1$ and~\eqref{eq:product
    log}.
\end{proof}

\begin{proof}[Proof of Theorem~\ref{thm:precise MDC result} (2)] 
  Let $i \colon H \to G$ be the inclusion of a subgroup, and let
  $(X,\mu)$ be a standard probability space endowed with an
  essentially free standard $H$-action.  Let $i_!X$ be the coinduction
  of $X$, i.e. the $G$-space $\map_H(G,X)$, on which $g \in G$ acts
  from the left by composition with the $G$-map $r_{g^{-1}} \colon G
  \to G,~ g_0 \mapsto g_0g^{-1}$. By choosing a set theoretic section
  $s \colon G/H \to G$ of the projection with $s(1)=1$, we obtain a
  bijection
  \[i_!X \xrightarrow{\cong} \prod_{gH \in G/H} X,\quad\phi
  \mapsto\bigl(\phi(s(gH))\bigr)_{gH\in G/H}.\] We endow $i_!X$ with
  the structure of a standard probability space $(i_!X,\nu)$ by
  pulling back the product measure on $\prod_{gH \in G/H} X$.  This
  structure does not depend on the choice of $s$; the measure $\nu$ is
  $G$-invariant~\cite{gaboriau-examples}*{3.4}.

  Let $\pr:i_!X \rightarrow X$ be the map sending $\phi$ to $\phi(1)$.
  We obtain a trace preserving, $H$-equivariant $\ast$-homomorphism
  $\pr^* \colon L^{\infty}(X) \to L^{\infty}(i_!X)$ by composition
  with $\pr$.  Thus we obtain a trace preserving $\ast$-algebra
  homomorphism
  \[u \colon L^{\infty}(X)\ast H \to L^{\infty}(i_!X)\ast G,\quad\sum_{h
    \in H} \lambda_h \cdot h ~ \mapsto \sum_{h\in H}\pr^*(\lambda_h)
  \cdot h,\] which extends to the von Neumann algebras $u \colon
  \caln(H\action X) \to \caln(G\action
  i_!X)$~\citelist{\cite{dixmier}\cite{sauer-promotion}*{Theorem~1.47}}.

  Let $A \in M(m\times n;L^{\infty}(X,\bbZ)\ast H)$.  Let $u_*A\in
  M(m\times n;L^\infty(i_!X,\bbZ)\ast G)$ be the matrix obtained from
  $A$ by applying elementwise the ring homomorphism $u$.  By
  hypothesis, ${\det}_{\caln(G\action i_!X)}(u_*A)\ge 1$.  The
  assertion follows from $\det_{\caln(H\action
    X)}(A)=\det_{\caln(G\action i_!X)}(u_*A)$
  (Remark~\ref{rem:inclusion of vN algebras}).
\end{proof}

\subsection{Extensions with amenable quotients}\label{sub:extensions}

\begin{proof}[Proof of Theorem~\ref{thm:precise MDC result} (3)]
  Let $G'\subset G$ be a subgroup. Then $H'=H\cap G'$ is normal in
  $G'$, and $G'/H'$ injects into $G/H$, thus $G'/H'$ is also
  amenable. Obviously, $H'\action X$ satisfies MDC, if $H\action X$
  does. We have to show that $\determ_{\caln(G\action X)}(A)\ge 1$ for
  every matrix $A$ over the ring $L^\infty(X)\ast G$. Taking
  $G'\subset G$ to be the subgroup generated by the finitely many
  elements of $G$ appearing in such $A$, it is enough to show that
  $G'\subset X$ satisfies MDC for every finitely generated subgroup
  $G'\subset G$. By our first remark, we thus may and will assume that
  $G$ is finitely generated.

  Let $p\colon G\to G/H$ be the projection.  We choose a
  left-invariant word-metric $d$ on the finitely generated group
  $G/H$.  For $R>0$ and a subset $Z\subset G/H$ we define
  \[
  N_R(Z) = \bigl\{x \in G/H \mid d(x,Z) \le R \text{ and } d(x,G/H-Z)
  \leq R\bigr\}.
  \]
  By amenability (compare~\cite{Luc2}*{Lemma~13.40 on p.~469}) there
  exists an increasing exhaustion of $G/H$ by finite subsets $Z_1
  \subset Z_2 \subset Z_3\subset\dots\subset G/H$ (\emph{F{\o}lner
    exhaustion}) such that for all $R>0$ and $\epsilon>0$ we find $N
  \in \N$ satisfying $|N_R(Z_n)| \le\epsilon \cdot |Z_n|$ for all $n
  \geq N$.  Let $S$ be a transversal for $H$ in $G$. We set $S_n = \{
  s \in S \big| sH \in Z_n \} \subset G$. We have
  $\abs{S_n}=\abs{Z_n}$.  Let $p_n:L^2(G \curvearrowright X)\to
  L^2(G\action X)$ be the projection onto the closure of the span of
  $p^{-1}(Z_n)$ and $L^2(X)$, i.e.
  \[
  p_n\big( \sum_{g \in G} r_g \cdot g \big) =
  \sum_{\stackrel{\scriptstyle g \in G}{gH \in Z_n}} r_g \cdot g.
  \]
  The map $p_n$ is not $L^\infty(X)\ast G$-equivariant in general but
  $L^\infty(X)\ast H$-equi\-variant. 

We remark that in the group case ($X=\{\ast\}$) a more general statement, where 
$H$ is not necessarily normal, is stated 
in~\cite{Sch}*{Section~4} 
and~\cite{Luc2}*{Proposition~13.93 on p.~469}. However, the proofs 
of these statements are not correct: the mistakes are related to the  
equivariance of the map $p_n$ above for which normality of $H$ is essential. 

Notice that we have an isometric
  $L^\infty(X)\ast H$-equivariant isomorphism
  \[
  u_n\colon L^2(H \curvearrowright X)^{|S_n|} \to \im(p_n), \quad
  (f_s)_{s \in S_n} \mapsto \sum_{s \in S_n} f_s \cdot s.
  \]
  Let $A \in M(d \times d';L^\infty(X,\Z)\ast G)$. In order to show
  $\det_{\caln(G\action X)}(A)\ge 1$ we may assume that $d=d'$ and $A$
  is positive (Lemma~\ref{lem:positive is enough}).  Consider the
  operator
  \[L^2(H \curvearrowright X)^{d|S_n|}\xrightarrow{\oplus
    u_n}\im(p_n)^d\xrightarrow{r_A}L^2(G\action
  A)^d\xrightarrow{\oplus p_n}\im(p_n)^d\xrightarrow{\oplus u_n^{-1}}
  L^2(H \curvearrowright X)^{d|S_n|}.\] It is easy to see that this
  operator is given by right multiplication with a positive matrix
  $A_n\in M(d \cdot |Z_n| \times d \cdot |Z_n|;L^\infty(X,\Z)\ast H)$.
  By hypothesis, we have $\det_{\caln(H\action X)}(A_n)\ge 1$ for
  every $n\ge 1$.
  Since $\|p_n\| \leq 1$ holds for all $n \in \N$, we conclude $\|r_{A_n}\| \leq \|r_A\|$.\\
  By Lemma~\ref{lem:Approximation Lemma} it suffices to show that
  \[
  \lim_{n \to \infty} \frac{\tr_{{\mathcal N}(H \curvearrowright
      X)}(A_n^m)}{d \cdot |Z_n|} = \frac{\tr_{{\mathcal N}(G
      \curvearrowright X)}(A^m)}{d} \mbox{ for all } m \in \N.
  \]
  This is proven in the case $X = \{\ast\}$ in~\cite{Luc2}*{Lemma
    13.42 on page 470}, and the proof is essentially the same in our
  setting.
\end{proof}

\subsection{Extensions with finite kernels} 
\label{sub:extensions_with_finite_kernels}

\begin{proof}[Proof of Theorem~\ref{thm:precise MDC result} (4)]
  Let $Q\action X$ be an essentially free, standard action. Let $G$
  act on $X$ via $p:G\rightarrow Q$. In the sequel, we write
  $F=l^2(L^\infty(X)\astvn G)$ for the GNS-construction of the von
  Neumann algebra $L^\infty(X)\astvn G$.

  Let $N_K \in G$ be the element $\sum_{g \in K} g$. Consider the
  Hilbert $L^\infty(X)\astvn G$-morphism
  \[r_{|K|^{-1}\cdot N_K}\colon F \to F.\] Then $r_{|K|^{-1}\cdot
    N_K}$ is an orthogonal projection with the $K$-fixed points $F^K$
  as image since $K \subseteq G$ is normal and $K$ acts trivially on
  $X$. In particular, $F^K$ is a finitely generated Hilbert
  $L^\infty(X)\astvn G$-module.  Define an isometric bijective
  operator
  \[v \colon L^2(Q\action X) \xrightarrow{\cong} F^K, ~\sum_{q \in Q}
  \lambda_q \cdot q\mapsto |K|^{-1/2} \cdot \sum_{g \in G} \lambda_{p(g)}
  \cdot g.\] We also set
  \[w = v^{-1} \circ r_{|K|^{-1}\cdot N_K} \colon F \to L^2(Q \action
  X).\] For every $g\in G$, $f\in L^\infty(X)$, and $a\in L^2(Q\action
  X)$ we have
  \begin{align}\label{eq:rules}
    v(r_{p(g)}(a))&=r_g(v(a)),\\
    v(r_f(a))&=r_f(v(a)).\notag
  \end{align}
  This, in particular, implies that the image of a $\caln(Q\action
  X)$-invariant subspace $L^2(Q\action X)^d$ under
  $\diag(v):L^2(Q\action X)^d\to (F^K)^d$ is $L^\infty(X)\astvn
  G$-invariant. Upon choosing an isometric embedding into
  $L^2(Q\action X)^d$ and thus into $(F^K)^d$ via $\diag(v)$, we can
  equip every finitely generated Hilbert $\caln(Q\action X)$-module
  $V$ with the structure of a finitely generated Hilbert
  $L^\infty(X)\astvn G$-module. We denote $V$ endowed with this new
  structure by $\res_p V$. Because of~\eqref{eq:rules} these module
  structures are related by $g\phi(x)=p(g)x$ and $f\phi(x)=fx$, where
  we denote the identity $V\rightarrow\res_p V$ by $\phi$ for better
  distinction. From this we also see that the module structure on
  $\res_p V$ does not depend on the chosen embedding $V\hookrightarrow
  L^2(Q\action X)^d$ and that every morphism $f \colon V \to W$ of
  finitely generated Hilbert $\caln(Q \action X)$-modules induces a
  morphism $\res_p f \colon \res_p V \to \res_p W$ of finitely
  generated Hilbert $L^\infty(X)\astvn G$-modules by the same map.
  Next we show that
  \begin{equation}
    \tr_{L^\infty(X)\astvn G}(\res_p f)=\frac{1}{|K|} \cdot 
    \tr_{\caln(Q \action X)}(f).
    \label{tr_G(res_p(f)) = |K|^{-1} tr_Q(f)}
  \end{equation}
  It suffices to treat the case $V = L^2(Q \action X)$. Let $e_G \in
  F$ and $e_Q \in L^2(Q \action H)$ be the elements given by the unit
  element in the rings $L^{\infty}(X) \ast G$ and $L^{\infty}(X) \ast
  Q$. Write $f(e_Q) = \sum_{q \in Q} \lambda_q \cdot q$. Then
  \[w^*\circ f \circ w(e_Q) ~ = ~ \frac{1}{|K|} \cdot \sum_{g \in G}
  \lambda_{p(g)}\cdot g.\] This implies
  \begin{align*}
    \tr_{\caln(Q \action X)}(f) &=\langle f(e_Q), e_Q \rangle_{L^2(Q
      \action X)}
    \\
    &=\langle 1,\lambda_{e_Q} \rangle_{L^2(X)}
    \\
    &=\abs{K}\cdot \langle w^*\circ f \circ w(e_G), e_G \rangle_{F}
    \\
    &=\abs{K}\cdot \tr_{L^\infty(X)\astvn G}(\res_p f).
  \end{align*}
  Hence \eqref{tr_G(res_p(f)) = |K|^{-1} tr_Q(f)} follows.

  If $\{E_{\lambda} \mid \lambda \in \bbR\}$ is the spectral family of
  $f \colon V \to V$, then $\{\res_p E_{\lambda} \mid \lambda \in
  \bbR\}$ is the spectral family of $\res_p f \colon \res_p V \to
  \res_pV$. Hence~\eqref{tr_G(res_p(f)) = |K|^{-1} tr_Q(f)}
  successively yields that
  \begin{align}\label{det(res_p f) = det(f)^{1/|K|}}
    F(\res_p f)&=\frac{1}{|K|} \cdot F(f),\notag \\
    \determ_{L^\infty(X)\astvn G}(\res_p f) &= \left(\determ_{\caln(Q
        \action X)}(f)\right)^{1/|K|}.
  \end{align}

  Let $A \in M(d' \times d, L^\infty(X,\Z)\ast Q)$. We have to show
  that $\det_{\caln(Q\action X)}(A)\ge 1$.  By Lemma~\ref{lem:positive
    is enough} we may and will assume that $d'=d$ and $A$ is positive.

  Let $n\in\bbN$. We get a morphism $\res_p r_{A^n}\colon \res_p L^2(Q
  \action X)^d \to \res_p L^2(Q\action X)^d$ of finitely generated
  Hilbert $L^\infty(X)\astvn G$-modules. We have the orthogonal sum decomposition
  \[F = \underbrace{\im(r_{|K|^{-1} \cdot N_K})}_{=F^K} \oplus
  \im(r_{1- |K|^{-1} \cdot N_K}).\] Consider the morphism
  \[w^* \circ r_{A^n} \circ w \oplus \id_{\im(r_{1-|K|^{-1} \cdot
      N_K})^d} \colon F^d \to F^d.\]

  We conclude from \eqref{det(res_p f) = det(f)^{1/|K|}} and
  \cite{Luc2}*{Theorem~3.14~(1) on p.~128 and Lemma~3.15~(7) on
    p.~130} that
  \begin{align}\label{det(r_A^n^(2)}
    \determ_{\caln(Q \action X)}(r_{A})^n &=
    \determ_{\caln(Q \action X)}(r_{A^n})\notag \\
    &=\determ_{L^\infty(X)\astvn G}(\res_p r_{A^n})^{\abs{K}}\notag\\
    &=\determ_{L^\infty(X)\astvn G}(w^* \circ r_{A^n} \circ
    w)^{\abs{K}}
    \cdot\determ_{L^\infty(X)\astvn G}\bigl(\id_{\im(r_{1-|K|^{-1} \cdot N_K})^d}\bigr)^{\abs{K}}\notag\\
    &=\determ_{L^\infty(X)\astvn G} \bigl(w^* \circ r_{A^n} \circ w
    \oplus \id_{\im(r_{1-|K|^{-1} \cdot N_K})^d} \bigr)^{\abs{K}}.
  \end{align}

  For $u = \sum_{q \in Q} \lambda_q \cdot q$ in $L^{\infty}(X,\bbZ)
  \ast Q$ let $s(u) \in L^{\infty}(X,\bbZ) \ast G$ be the
  element $\sum_{g \in G} \lambda_{p(g)} \cdot g$.
  Define $B = (b_{i,j})_{i,j} \in M(d\times d,L^{\infty}(X,\bbZ) \ast
  G)$ to be the matrix obtained from $A^n = (a_{i,j})$ by setting
  $b_{i,i} = s(a_{i,i} - 1)$ and $b_{i,j} = s(a_{i,j})$ if $i\ne j$. 
  One easily verifies that
  \[\frac{1}{|K|} \cdot r_{|K| \cdot I_d + B} =
  w^* \circ r_{A^n} \circ w \oplus \id_{\im(r_{1-|K|^{-1} \cdot
      N_K})^d},\] where $I_d$ is the identity matrix in $M(d\times
  d,L^{\infty}(X,\bbZ) \ast G)$.  Notice that $|K| \cdot I_d + B$ lies
  in $M(d\times d,L^{\infty}(X,\bbZ) \ast G)$; thus, by hypothesis and
  Lemma~\ref{lem:reduction for non-free actions},
  $\determ_{L^\infty(X)\astvn G}\left(r_{|K| \cdot I_d + B}\right) \ge
  1$.  We conclude from~\cite{Luc2}*{Theorem 3.14 (1) on page 128}
  that
  \begin{align}\label{det (r_res_pA = |K|^-d cdot det(r_1-B)}
    \determ_{L^\infty(X)\astvn G} \left(w^* \circ r_{A^n}\circ w
      \oplus \id_{\im(r_{1-|K|^{-1} \cdot N_K})^d} \right)&=
    \determ_{L^\infty(X)\astvn G}\left(\frac{1}{|K|} \cdot r_{|K|
        \cdot I_d + B}\right)
    \notag\\
    &=\frac{1}{|K|^d} \cdot\determ_{L^\infty(X)\astvn G}\left(r_{|K|
        \cdot I_d + B}\right)
    \notag\\
    &\ge\frac{1}{|K|^d}.
  \end{align}
  We conclude from~\eqref{det(r_A^n^(2)} and~\eqref{det (r_res_pA =
    |K|^-d cdot det(r_1-B)} that
  \[\determ_{\caln(Q\action X)}(r_A)\ge\abs{K}^{-d\abs{K}/n}\]
  holds for every $n\in\bbN$. Hence ${\det}_{\caln(Q\action X)}(r_A)
  \ge 1$.
\end{proof}

\subsection{Colimits}\label{sub:colimits}
Throughout this subsection, we consider a directed system of groups
$\{G_i \mid i \in I\}$ over the directed set $I$.  Denote its colimit
by $G = \colim_{i \in I} G_i$.  Let $\psi_i\colon G_i \to G$ for $i
\in I$ and $\psi_{i,j}\colon G_i \to G_j$ for $i,j \in I, i\le j$, be
the structure maps. We do not require that $\psi_i$ or $\psi_{i,j}$
are injective.

\begin{proof}[Proof of Theorem~\ref{thm:precise MDC result} (5)]
  Let $G\action (X,\mu)$ be an essentially free, standard action.
  Every $G_i$ acts on $X$ via $\psi_i$ (but not necessarily free).  We
  obtain a trace-preserving ring $\ast$-homomorphism
  \[
  \alpha_i\colon L^\infty(X)\ast G_i\to L^\infty(X)\ast G,~\sum_{h\in
    G_i} l_h\cdot h \mapsto \sum_{h\in G_i} l_h\cdot\psi_i(h).
  \]
  Let $A \in M(m\times n,L^\infty(X;\bbZ) \ast G)$.  Write $A= \sum_{g
    \in G}f_g g$ with $f_g \in M(m\times n,L^{\infty}(X,\bbZ))$. Let
  $i_0\in I$ be such that for every $i\ge i_0$ the implication
  \[
  f_g\ne 0\Rightarrow g\in\Im(\psi_i)
  \]
  holds. Let $V=\{g\in G\vert f_g\ne 0\}$. For every $g\in V$ let
  $g^{(i)}\in G_i$ be a preimage of $g$. Let $A_i=\sum_{g\in V}
  f_g\cdot g^{(i)}$ and $\Delta_i=A_iA_i^\ast$.  We have
  $\alpha_i(A_i)=A$. From Lemma~\ref{lem: uniform bound on the norm}
  it is clear that there is a uniform bound of the operator norms of
  the $\Delta_i\in M(n\times n,L^\infty(X)\astvn G_i)$. By hypothesis
  and Lemma~\ref{lem:reduction for non-free actions} we have
  \[
  \determ_{L^\infty(X)\astvn G_i}\bigl(A_i\bigr)\ge 1.
  \]
  Let $m\ge 1$. The assertion would follow from
  Lemma~\ref{lem:Approximation Lemma} provided we show that
  \begin{equation}\label{eq:final argument for colimit}
    \tr_{\caln(G\action X)}(\Delta^m)=
    \lim_{i\in I}\tr_{L^\infty(X)\astvn G_i}(\Delta_i^m). 
  \end{equation}
  We can find $i_1\ge i_0$ in $I$ such that for every $h\in G_{i_1}$
  with $l_h\ne 0\in M(n\times n,L^\infty(X_{i_1}\times X))$ in the
  finite linear combination $\Delta_{i_1}^m=\sum_h l_hh$ we have the
  implication
  \[
  \psi_{i_1}(h)=1\Rightarrow h=1.
  \]
  For $i\ge i_1$ the right hand side of~\eqref{eq:final argument for
    colimit} is stationary and equal to the left hand side.
\end{proof}

\subsection{Bernoulli actions}\label{sub:sofic groups}

\begin{proof}[Proof of Theorem~\ref{thm:precise MDC result} (6)]
  For any countable set $A$ denote by $\mu_A$ the product measure
  $\bigotimes_{a\in A}\mu$ on $X^A$. The $\sigma$-algebra of Borel
  sets in $X^G=\prod_{g \in G} X$ is the $\sigma$-algebra ${\mathcal
    S}$ generated by
  \[
  {\mathcal B} = \bigl\{\prod_{g \in G} U_g \, \big| \, U_g = X
  \mathrm{\; for \; almost \; all \;} g \in G\bigr\}.
  \]
  Let ${\mathcal A}$ be the algebra generated by ${\mathcal B}$. We
  say that a measurable function $f:X\rightarrow\bbZ$ is
  $\cala$-measurable if $f^{-1}(z)\in\cala$ for every $z\in\bbZ$.  Any
  set $M \in {\mathcal A}$ can be written as $M = \cup_{k=1}^n M_k$
  with disjoint sets $M_k \in {\mathcal B}$. The sets $M \in {\mathcal
    A}$ have the property that there exists a finite subset $F
  \subseteq G$ with $M = \pr_F^{-1}(\pr_F(M))$ where $\pr_F\colon
  \prod_{g \in G} X\to \prod_{g \in F} X$ is the projection onto the
  components of $F$. We denote with $F(M) \subseteq G$ the smallest
  subset with this property.  Let $R$ be the ring of all bounded,
  $\cala$-measurable functions $X^G\rightarrow\bbZ$.  For $f\in R$, we
  denote with $F(f)\subset G$ the (finite) union
  $F(f)=\bigcup_{z\in\bbZ} F(f^{-1}(z))$. Since any non-empty set in
  $\cala$ has positive measure, the ring $R$ embeds into
  $L^\infty(X^G,\bbZ)$. Since $R$ is $G$-invariant, we obtain a
  subring $R\ast G\subset L^\infty(X^G,\bbZ)\ast G$ that is closed
  under the involution.

  Let $\psi_i\colon G\rightarrow G_i$, $i\in I$, be the structure maps
  of the limit.  The map $\psi_i$ induces a measurable map
  \[\alpha_i:X^{G_i}\rightarrow X^G,
  ~\alpha_i((x_h))_g=x_{\psi_i(g)}.\] If $\psi_i\vert_{F(M)}$ is
  injective for $M\in\cala$, then
  $\mu_{G_i}(\alpha_i^{-1}(M))=\mu_G(M)$. Since every $f\in R$ is a
  finite linear combinations of characteristic functions of sets in
  $\cala$, we also obtain that
  \begin{equation}\label{eq:integral if injective}
    \int_{X^{G_i}}f\circ\alpha_i(z)d\mu_{G_i}(z)=\int_{X^G} f(x)d\mu_G(x)
  \end{equation}
  provided $\psi_i\vert_{F(f)}$ is injective.

  We obtain a ring homomorphism which respects the involutions:
  \[
  \beta_i\colon R\ast G\to L^\infty\bigl(X^{G_i}\bigr)\ast
  G_i,~\sum_{g\in G}f_g\cdot g\mapsto \sum_{g\in G}(f_g\circ
  \alpha_i)\cdot\psi_i(g).
  \]
  By applying this homomorphism entry-wise we obtain a ring
  homomorphism $M(d\times d',R\ast G)\to M(d\times
  d',L^\infty(X^{G_i})\ast G_i)$ that we denote by the same name.

\begin{lemma}\label{lem:sub-lemma; limit of elements}
  Let $A\in M(d\times d',R\ast G)$ and $\Delta=AA^\ast$.  Let
  $\Delta_i=\beta_i(A)\beta_i(A)^\ast$. Let $m\ge 1$.  Then
  \begin{equation}\label{eq:sofic limit of matrices}
    \tr_{\caln(G\action X^G)}(\Delta^m)=
    \lim_{i\rightarrow\infty}\tr_{\caln(G_i\action X^{G_i})}\bigl(\Delta_i^m\bigr),  
  \end{equation}
  and we have $\det_{\caln(G\action X^G)}(A)\ge 1$.
\end{lemma}

\begin{proof}[Proof of lemma]
  By considering the matrix entries separately, the assertion reduces
  to showing that for any set of $2m$ elements $a_1,\dots, a_m,
  b_1,\dots,b_m\in R\ast G$ we have
  \begin{equation*}
    \tr_{\caln(G\action X^G)}(a_1b_1^\ast a_2b_2^\ast\cdots a_mb_m^\ast)=
    \lim_{i\rightarrow\infty}
    \tr_{\caln(G_i\action X^{G_i})}
    \bigl(\beta_i(a_1)\beta_i(b_1)^\ast\cdots\beta_i(a_m)\beta_i(b_m)^\ast\bigr). 
  \end{equation*}
  Since $\beta_i$ is a ring $\ast$-homomorphism, the assertion reduces
  further to showing that for $f\in R$ and $g\in G$ we have
  \begin{equation*}
    \tr_{\caln(G\action X^G)}(f\cdot g)=\lim_{i\to\infty}\tr_{\caln(G_i\action X^{G_i})}(\beta_i(f\cdot g)). 
  \end{equation*}
  Choose $i_0\in I$ such that $\psi_i\vert_{F(f)\cup\{1,g\}}$ is
  injective for $i\ge i_0$.  Then~\eqref{eq:integral if injective}
  yields that
  \[\tr_{\caln(G\action X^G)}(f\cdot g)=\tr_{\caln(G_i\action
    X^{G_i})}(\beta_i(f\cdot g))\] for $i\ge i_0$.  This concludes the
  proof of~\eqref{eq:sofic limit of matrices}.  By Lemma~\ref{lem:
    uniform bound on the norm} there is an upper bound
  $\norm{r_{\Delta_i^m}}$ that is independent of $i$.  By hypothesis,
  $\determ_{\caln(G_i\action X^{G_i})}(\beta_{i}(A)^m)\ge 1$.
  Finally, Lemma~\ref{lem:Approximation Lemma} implies that
  $\det_{\caln(G\action X^G)}(A)\ge 1$.
\end{proof}

We need a general fact before we can complete the proof: Let
$f_1,f_2\dots, f_m\in L^\infty(X)$. For every $1\le i\le m$ let
$f_i^{(n)}\in L^\infty(X)$ be a sequence such that there is a constant
$C>0$ with $\fixnorm{f_i^{(n)}}_{L^\infty(X)}<C$ and
$\lim_{n\to\infty}\fixnorm{f_i-f_i^{(n)}}_{L^1(X)}=0$.  Then:
\begin{equation}\label{eq:l1 limit of functions}
  \lim_{n\rightarrow\infty}\fixnorm{f_1\cdots f_m-f_1^{(n)}\cdots f_m^{(n)}}_{L^1(X)}=0.
\end{equation}
This follows from an iterated application of the corresponding
assertion for $m=2$. For $m=2$ we have:
\begin{align*}
  \int_X&\fixabs{f_1(x)f_2(x)-f_1^{(n)}(x)f_2^{(n)}(x)}d\mu(x)\\
  &\le\fixnorm{f_2}_{L^\infty(X)}\int_X\fixabs{f_1(x)-f_1^{(n)}(x)}d\mu(x)+
  \fixnorm{f_1^{(n)}}_{L^\infty(X)}\int_X\fixabs{f_2(x)-f_2^{(n)}(x)}d\mu(x)\\
  &\le
  \fixnorm{f_2}_{L^\infty(X)}\cdot\fixnorm{f_1-f_1^{(n)}}_{L^1(X)}+C\cdot\fixnorm{f_2-f_2^{(n)}}_{L^1(X)}\xrightarrow{n\rightarrow\infty}
  0.
\end{align*}

Now we can complete the proof of Theorem~\ref{thm:precise MDC result}
(6).  For any $S \in {\mathcal S}$ and any $\epsilon > 0$ there exists
a set $M \in {\mathcal A}$ with $\mu_G(S \triangle M)<\epsilon$. Since
every element in $f\in L^\infty(X,\bbZ)$ is a finite linear
combination of characteristic functions, there exist functions
$f^{(n)}\in R$ associated to $f$ such that
$\norm{f^{(n)}}_{L^\infty(X)}\le\norm{f}_{L^\infty(X)}$ for all $n\ge
1$ and $\norm{f-f^{(n)}}_{L^1(X)}<1/n$.  Let $A\in M(d\times
d',L^\infty(X)\ast G)$. Let $B^{(n)}\in M(d\times d',R\ast G)$ be the
matrix obtained from $A$ by replacing each entry $\sum f_g\cdot g$ with
$\sum f_g^{(n)}\cdot g$.  Let $\Delta=AA^\ast$ and
$\Delta_n=B^{(n)}(B^{(n)})^\ast$.  For every $m\ge 1$ we have
\begin{equation}\label{eq:approximation by matrices in R}
  \tr_{\caln(G\action X)}(\Delta^m)=
  \lim_{n\rightarrow\infty}\tr_{\caln(G\action X)}(\Delta_n^m). 
\end{equation}
By considering matrix entries separately as in the proof of
Lemma~\ref{lem:sub-lemma; limit of elements}, this easily follows
from~\eqref{eq:l1 limit of functions} and the $G$-invariance of the
measure.  By Lemma~\ref{lem: uniform bound on the norm},
$\norm{r_{\Delta_n^m}}$ has an upper bound independent of
$n$. Lemma~\ref{lem:sub-lemma; limit of elements} and
Lemma~\ref{lem:Approximation Lemma} complete the proof.
\end{proof}

\section{Proofs of Theorems~\ref{thm:MDC implies det class for ME
    groups} and~\ref{thm:invariance under uniform ME}}
\label{sec:uniform measure equivalence invariance}

The following two results are proved in first author's
book~\cite{Luc2} for group von Neumann algebras. We need them for
arbitrary finite von Neumann algebras; the proof translates literally
to the general case.

\begin{proposition}[\cite{Luc2}*{Theorem 3.35 (5) on p.~142}]\label{prop-cone}
  Let $\cala$ be a finite von Neumann algebra with a fixed trace.  Let
  $\phi_*\colon C_* \to D_*$ be a chain map of finitely generated
  Hilbert $\cala$-chain complexes.  Suppose that $b_n^{(2)}(C_*) =
  b_n^{(2)}(D_*) = 0$ and $\det_{\cala}(c_n)$, $\det_{\cala}(d_n) > 0$
  for all $n \in \Z$.  Then the mapping cone $\cone_*(\phi_*)$ is also
  a finitely generated Hilbert $\cala$-chain complex with vanishing
  $L^2$-Betti numbers and positive determinants of his differentials.
  Moreover, we obtain the equation
  \[
  \rho^{(2)}\big(\cone_*(\phi_*)\big) = \rho^{(2)}(D_*) -
  \rho^{(2)}(C_*).
  \]
\end{proposition}

\begin{lemma}[\cite{Luc2}*{Lemma~3.41 on p.~146}]\label{lem-acyclic}
  Let $C_*$ be a finitely generated Hilbert $\cala$-chain complexes
  and $\gamma_*$ a chain contraction. Then $b_n^{(2)}(C_*) = 0$ for
  all $n \in \Z$.  If $\det_{\cala}(c_n) > 0$ for all $n \in \Z$ then
  \[
  \rho^{(2)}(C_*) = \ln\determ_{\cala}\Bigl((c_* + \gamma_*)_{\odd}:
  \oplus_{n \in \Z}~C_{2n+1} \to \oplus_{n \in \Z}~C_{2n}\Bigr).
  \]
\end{lemma}

\begin{definition}\label{def:proj-based}
  Let $R$ be a ring with involution. A \emph{finitely generated based
    free $R$-module} is a finitely generated free $R$-module together
  with an isomorphism $M\cong R^n$. A \emph{finitely generated based
    projective $R$-module} $P$ is a finitely generated projective
  $R$-module $P$ together with an isomorphism $P\cong R^n A$ where
  $A\in M(n\times n,R)$ satisfies $A^2=A$ and $A^\ast=A$.
\end{definition}

\begin{remark}\label{rem: bases and restriction}
  Consider a standard action $G\action X$ and a measurable subset
  $A\subset X$ such that there is a finite subset $S\subset G$ with
  $S\cdot A=X$ up to null sets.  We can find (and fix for the
  following discussion) measurable subsets $A_g\subset A$ for each
  $g\in S$ such that the sets $gA_g$ partition $X$.  We obtain the
  isomorphism of left $\chi_A \bigl(L^\infty(X,\bbZ)\ast G\bigr)\chi_A$-modules 
  \begin{equation}\label{eq:fixed iso}
    \phi:\chi_A \bigl(L^\infty(X,\bbZ)\ast G\bigr)\xrightarrow{\cong} 
    \bigoplus_{g\in S}\chi_A \bigl(L^\infty(X,\bbZ)\ast G\bigr)\chi_{A_g},~ \phi(x)=\sum_{g\in S}xg\chi_{A_g}. 
  \end{equation}
  If $\psi:F\xrightarrow{\cong} (L^\infty(X,\bbZ)\ast G)^n$ is a
  finitely generated based free module, then $\chi_A F$ becomes a
  finitely generated based projective module over the ring $\chi_A
  L^\infty(X,\bbZ)\ast G\chi_A$ by
  \[\chi_A F\xrightarrow{\chi_A\psi}\bigl(\chi_A L^\infty(X,\bbZ)\ast
  G\bigr)^n \xrightarrow{\oplus_{i=1}^n\oplus_S\phi} \bigl(\chi_A
  L^\infty(X,\bbZ)\ast G\chi_A\bigr)^{\abs{S}n} Q,
  \]
  where is $Q$ is the projection matrix
  \[
  Q=\diag((\chi_{A_g})_{g\in S})\oplus\dots\oplus
  \diag((\chi_{A_g})_{g\in S}).
  \]

  Let $P\xrightarrow{\cong} (\chi_A L^\infty(X,\bbZ)\ast G\chi_A)^n
  Q'$ be a finitely generated based projective $\chi_A
  L^\infty(X,\bbZ)\ast G\chi_A$-module.  Then $L^2(G\action X\vert
  A)\otimes_{\chi_A L^\infty(X,\bbZ)\ast G\chi_A} P$ is isomorphic to
  the image of the orthogonal projection
  \[L^2(G \curvearrowright X|_A)^n\rightarrow L^2(G \curvearrowright
  X|_A)^n, x\mapsto xQ',
  \]
  and obtains the structure of a Hilbert $\caln(G\action X\vert
  A)$-module from this isomorphism.
	
  Let $F$ be a finitely generated based free module over the ring
  $L^\infty(X,\bbZ)\ast G$. Since $\chi_A$ is full, there is an
  obvious isomorphism
  \begin{equation}\label{eq:two Hilbert structures}
    \chi_A L^2(G\action X)\otimes_{L^\infty(X,\bbZ)\ast G}F\cong 
    \chi_A L^2(G\action X)\chi_A\otimes_{\chi_A L^\infty(X,\bbZ)\ast G\chi_A}\chi_A F.  
  \end{equation}
  If $V$ is any Hilbert $\caln(G\action X)$-module, then $\chi_A V$
  becomes a Hilbert $\caln(G\action X\vert_A)$-module
  (see~\cite{jones}*{pp.~19-27}).  Since $L^2(G\action
  X)\otimes_{L^\infty(X,\bbZ)\ast G}F$ is a Hilbert $\caln(G\action
  X)$-module through the free basis of $F$, the left hand side
  in~\eqref{eq:two Hilbert structures} is a Hilbert $\caln(G\action
  X\vert_A)$-module. On the other hand, the right hand side
  in~\eqref{eq:two Hilbert structures} is a Hilbert $\caln(G\action
  X\vert_A)$-module since $\chi_A F$ is based projective. The
  isomorphism~\eqref{eq:two Hilbert structures} is an isomorphism of
  Hilbert $\caln(G\action X\vert_A)$-modules.
\end{remark}

\begin{lemma}\label{lem:MDC for projective modules} 
  Let $G\action (X,\mu)$ be an essentially free, standard action and
  $A\subset X$ a measurable subset such there is a finite subset
  $F\subset G$ with $FA=X$ up to null sets. Then the following
  statements are equivalent:
  \begin{enumerate}
  \item $G\action X$ satisfies MDC.
  \item For every homomorphism of finitely generated based projective
    $L^\infty(X,\bbZ)\ast G$-modules $f\colon P\rightarrow Q$ the associated
    Hilbert $\caln(G\action X)$-morphism $L^2(G\action
    X)\otimes_{L^\infty(X,\bbZ)\ast G} f$ has determinant $\ge 1$.
  \item For every homomorphism of finitely generated based projective
    $\chi_AL^\infty(X,\bbZ)\ast G\chi_A$-modules $f\colon P\rightarrow Q$
    the associated Hilbert $\caln(G\action X\vert_A)$-morphism \\
    $L^2(G\action X\vert_A)\otimes_{\chi_AL^\infty(X,\bbZ)\ast
      G\chi_A} f$ has determinant $\ge 1$.
  \end{enumerate}
\end{lemma}

\begin{proof}
  (1) implies (2): Every homomorphism $f\colon P\rightarrow Q$ between
  finitely generated based projective $L^\infty(X,\bbZ)\ast G$-modules
  can be extended by zero to an homomorphism
  \[P\oplus P'\xrightarrow{f\oplus 0}Q\oplus Q'\] of finitely
  generated based free modules. We have $\determ_{\caln(G\action
    X)}(f)=\determ_{\caln(G\action X)}(f\oplus 0)$, and
  the latter is $\ge 1$ since $G$ satisfies MDC. \\
  \noindent (2) implies (3): By a similar argument, it is enough to
  show that, for every homomorphism of finitely generated based
  \emph{free} $\chi_AL^\infty(X,\bbZ)\ast G\chi_A$-modules
  $f:P\rightarrow Q$, the associated Hilbert $\caln(G\action
  X\vert_A)$-morphism $L^2(G\action
  X\vert_A)\otimes_{\chi_AL^\infty(X,\bbZ)\ast G\chi_A} f$ has
  determinant $\ge 1$. This is equivalent to: For every matrix $B\in
  M(m\times n,\chi_A L^\infty(X,\bbZ)\ast G\chi_A)$ we have
  $\determ_{\caln(\calr)}(B)\ge 1$. We can view such $B$ also as an
  element in $M(m\times n, L^\infty(X,\bbZ)\ast G)$. It is obvious
  (just check on polynomials) that the spectral density function of
  $B$ with respect to $\caln(G\action X)$ is just $\mu(A)$ times the
  spectral density function with respect to $\caln(G\action
  X\vert_A)$.  Assuming (2), we now see that $\determ_{\caln(G\action
    X\vert_A)}(B)\ge 1$ and
  \[
  \ln\determ_{\caln(G\action
    X)}(B)=\mu(A)\cdot\ln\determ_{\caln(G\action X\vert_A)}(B).
  \]
  \noindent (3) implies (1): Let $f:F\rightarrow F'$ be a homomorphism
  of finitely generated based free $L^\infty(X,\bbZ)\ast
  G$-modules. Since $\chi_A$ is full by Lemma~\ref{lem:full
    idempotent}, $\chi_A F$ and $\chi_A F'$ are finitely generated
  (based) projective $\chi_AL^\infty(X,\bbZ)\ast G\chi_A$-modules.  By
  (3) and Remark~\ref{rem: bases and restriction} the map
  \[
  \chi_A L^2(G\action X)\otimes f: \chi_A L^2(G\action
  X)\otimes_{L^\infty(X,\bbZ)\ast G}F\rightarrow \chi_A L^2(G\action
  X)\otimes_{L^\infty(X,\bbZ)\ast G}F'
  \]
  has determinant $\ge 1$. Thus, by Remark~\ref{rem:restriction of von
    Neumann algebras}, the map $L^2(G\action X)\otimes f$ has
  determinant $\ge 1$. By Lemma~\ref{lem:can restrict to crossed
    product ring} this suffices to show that $G\action X$ satisfies MDC.
\end{proof}

We omit the proof of the following lemma which is essentially the same
as that of Lemma~\ref{lem:MDC for projective modules}.

\begin{lemma}\label{lem:MDC for projective modules and groupoid ring} 
  Let $\calr$ be a standard equivalence relation on $(X,\mu)$. Let
  $A\subset X$ be a measurable subset such that $\chi_A$ is a full
  idempotent in $\bbZ\calr$.  Then the following statements are
  equivalent:
  \begin{enumerate}
  \item $\calr$ satisfies MDC.
  \item For every homomorphism of finitely generated based projective
    $\bbZ\calr$-modules $f\colon P\rightarrow Q$ the associated
    Hilbert $\caln(\calr)$-morphism $L^2(\calr)\otimes_{\bbZ\calr} f$
    has determinant $\ge 1$. ´
  \item $\calr\vert_A$ satisfies MDC.
  \item For every homomorphism of finitely generated based projective
    $\bbZ\calr\vert_A$-modules $f\colon P\rightarrow Q$ the associated
    Hilbert $\caln(\calr)$-morphism
    $L^2(\calr)\otimes_{\bbZ\calr\vert_A} f$ has determinant $\ge 1$.
  \end{enumerate}
\end{lemma}

\begin{proof}[Proof of Theorem~\ref{thm:MDC implies det class for ME groups}]
  By Theorem~\ref{thm:ME implies OE} there exist essentially free,
  standard actions $G\action (X,\mu_X)$ and $H\action (Y,\mu_Y)$ that
  are weakly orbit equivalent. Let $f\colon B \to A$ be a weak orbit
  equivalence between measurable subsets $A\subset X$ and $B\subset
  Y$. By~\cite{Fur1}*{Lemma~2.2} we may assume that both actions
  are ergodic. The map $f$ induces an isomorphism between the
  restricted orbit equivalence relations $\calr(G\action X)\vert_A$
  and $\calr(H\action Y)\vert_B$.  By hypothesis, $\calr(G\action X)$
  satisfies MCD. That $\calr(H\action Y)$ satisfies MDC, thus $H$ satisfies the determinant 
  conjecture, follows by applying Lemma~\ref{lem:MDC
    for projective modules and groupoid ring} twice provided that
  $\chi_A\in\bbZ\calr(G\action X)$ and $\chi_B\in\bbZ\calr(H\action
  Y)$ are full idempotents. For that, we prove in general that
  $\chi_A\in\bbZ\calr$ is a full idempotent if $\calr$ is an ergodic
  standard equivalence relation and $A\subset X$ is a measurable
  subset of positive measure: Let $X=\bigcup_{i=1}^{n+1} A_i$ be a
  Borel partition of $X$ such that $A_1,\dots, A_n$ all have the same
  measure as $A$ and $\mu(A_{n+1})\le\mu(A_n)$. By ergodicity
  and~\cite{Fur3}*{Lemma~2.1} there are measure isomorphisms
  $\phi_i:X\rightarrow X$ for each $1\le i\le n+1$ such that
  $(x,\phi_i(x))\in\calr$ for $x\in X$ and $\phi_i(A)=A_i$ for $1\le
  i<n$ and $A_{n+1}\subset\phi_{n+1}(A)$. The characteristic function
  $\chi_{\phi_i}$ of the graph of each $\phi_i$ is an element in
  $\bbZ\calr$.  From the properties of $\phi_i$ we obtain that
  \[
  1=\chi_X=\sum_{i=1}^n
  \chi_{\phi_i}^\ast\chi_A\chi_{\phi_i}+\chi_{\phi_{n+1}}^\ast\chi_A\chi_{\phi_{n+1}}\chi_{\im(\phi_{n+1})}.
  \]
  Thus, $\chi_A$ is a full idempotent.
\end{proof}

\begin{proof}[Proof of Theorem~\ref{thm:invariance under uniform ME}]
  Let $Z_G$ and $Z_H$ be finite CW-models of the classifying spaces
  $BG$ and $BH$, respectively.  By Theorem~\ref{thm:ME implies OE}
  there exist essentially free, standard actions $G\action (X,\mu_X)$
  and $H\action (Y,\mu_Y)$ that are uniformly weakly orbit equivalent.
  Let $f\colon B \to A$ be a uniform weak orbit equivalence between
  measurable subsets $A\subset X$ and $B\subset Y$.  It induces a
  trace-preserving ring isomorphism $f_*\colon \chi_A
  L^\infty(X,\Z)*G\chi_A \to \chi_B L^\infty(Y,\Z)*H \chi_B$
  (Lemma~\ref{lem:uniform OE induces ring iso}) which extends to an
  isometry $f_*: L^2(G \curvearrowright X|_A) \to L^2(H
  \curvearrowright Y|_B)$.  By assumption, $G$ satisfies MDC. Because
  of the previous ring isomorphism and by applying Lemma~\ref{lem:MDC
    for projective modules} twice it follows that $H\action Y$
  satisfies MDC; in particular, $H$ satisfies the determinant
  conjecture.

  The cellular chain complex $C^{\cell}_*({\tilde Z_G})$ of the
  universal cover $\tilde{Z_G}$ is a finite based free $\Z
  G$-resolution of $\Z$.  By successively applying~\eqref{eq: det
    under induction} and~\eqref{eq:det under restriction} to the
  differentials of $l^2(G)\otimes_{\bbZ G}C^{\cell}_*({\tilde Z_G})$,
  we obtain that
  \[
  \torsion\big(\chi_A L^2(G \curvearrowright X) \otimes_{\Z G}
  C^{\cell}_*({\tilde Z_G})\big) = \frac{\torsion(G)}{\mu_X(A)}.
  \]
  The $\chi_A L^\infty(X,\Z)*G\chi_A$-chain complex
  \[C(G)_* := \chi_A L^\infty(X,\Z)*G\otimes_{\Z G}
  C^{\cell}_*({\tilde Z_G})\] is a finite resolution by based
  projective modules of $\chi_A L^\infty(X,\Z) = L^\infty(A,\Z)$
  by~Lemma~\ref{lem:flat} and Lemma~\ref{lem:full idempotent}.
  Furthermore, we define:
  \[C(G)^{(2)}_* := L^2(G \curvearrowright X|_A) \otimes_{\chi_A
    L^\infty(X,\Z)*G\chi_A} C(G)_\ast.\] Since $\chi_A$ is full
  (Lemma~\ref{lem:full idempotent}),
  \[C(G)^{(2)}_*\cong \chi_A L^2(G \curvearrowright X) \otimes_{\Z G}
  C^{\cell}_*({\tilde Z_G}).\] Hence,
  \begin{equation*}
    \rho^{(2)}\big(C(G)^{(2)}_*\big) = 
    \frac{\rho^{(2)}(G)}{\mu_X(A)}.
  \end{equation*}
  By replacing $G\action X$ by $H\action Y$ and $A$ by $B$, we define
  $C(H)_\ast$ and $C(H)^{(2)}_\ast$ in an analogous fashion.  A
  similar discussion as before applies; so $C(H)_\ast$ is a finite
  projective resolution of $L^\infty(B)$ over the ring
  $\chi_BL^\infty(Y,\Z)*H\chi_B$, and
  \[
  \torsion\bigl(C(H)^{(2)}\bigr)=\frac{\torsion(H)}{\mu_Y(B)}.
  \]
  It remains to prove that
  \begin{equation}\label{eq:to prove}
    \torsion\bigl(C(G)^{(2)}\bigr)=\torsion\bigl(C(H)^{(2)}\bigr). 
  \end{equation}
  Via the ring isomorphism $f_\ast$ we obtain $\chi_B
  L^\infty(Y,\Z)*H\chi_B$-module structures on $C(G)_\ast$ and
  indicate this by writing $f_*(C(G)_*)$. So $f_*(C(G)_*)$ is a finite
  projective resolution of $L^\infty(B)$ over the ring $\chi_B
  L^\infty(Y,\Z)*H\chi_B$.  We have
  \[
  \torsion\big(L^2(H \curvearrowright Y|_B) \otimes_{\chi_B
    L^\infty(Y,\Z)*H\chi_B} f_*(C(G)_*) \big)=
  \torsion\big(C(G)^{(2)}_*\big).\] By the fundamental lemma of
  homological algebra there exists a $\chi_B
  L^\infty(Y,\Z)*H\chi_B$-linear chain homotopy equivalence
  \[
  h_\ast\colon f_\ast(C(G)_\ast) \to C(H)_\ast.
  \]
  Let $\cone_*(h_*)$ denote the mapping cone of
  $h_\ast$~\cite{Luc2}*{Definition~1.33 on p.~35}; each
  $\cone_i(h_\ast)$ is a finitely generated based projective $\chi_B
  L^\infty(Y,\Z)*H \chi_B$-module.  The chain complex $L^2(H
  \curvearrowright Y|_B) \otimes_{\chi_B L^\infty(Y,\Z)*H\chi_B}
  \cone_*(h_*)$ coincides with the cone of the Hilbert-$\caln(H\action
  Y\vert_B)$-chain map
  \[
  h^{(2)}_*\colon L^2(H \curvearrowright Y|_B) \otimes_{\chi_B
    L^\infty(Y,\Z)*H \chi_B} f_*(C(G)_*) \to C(H)^{(2)}_\ast
  \]
  induced by $h_*$.  Proposition \ref{prop-cone} implies
  \begin{align*}
    \rho^{(2)}\big( \cone_*(h^{(2)}_*) \big) &=
    \rho^{(2)}\big( C(H)^{(2)}_* \big) - \rho^{(2)}\big( L^2(H \curvearrowright Y|_B) \otimes_{\chi_B L^\infty(Y,\Z)*H \chi_B} f_*(C(G)_*) \big) \\
    &= \rho^{(2)}\big( C(H)^{(2)}_* \big) - \rho^{(2)}\big(
    C(G)^{(2)}_* \big).
  \end{align*}
  It remains that show that $\torsion(\cone_*(h_*))=0$: Since
  $\cone_*(h_*)$ is acyclic, there exists a chain contraction
  $\delta_*$.  Let us denote the differentials of the $\chi_B
  L^\infty(Y,\Z)*H \chi_B$-chain complex $\cone_*(h_*)$ by $c_*$ and
  the induced differentials of the Hilbert ${\mathcal N}(G)$-chain
  complex $\cone_*(h^{(2)}_*)$ by $c^{(2)}_*$. Lemma \ref{lem-acyclic}
  tells us that
  \[
  \rho^{(2)}\big(\cone_*(h^{(2)}_*)\big) = \ln\determ_{{\mathcal N}(H
    \curvearrowright Y|_B)}((c^{(2)}_* + \delta^{(2)}_*)_{\odd})
  \]
  with $\delta^{(2)}_*$ being the chain contraction induced by
  $\delta_*$.  Notice that
  \[
  (c_* + \delta_*)_{\odd}\colon \oplus_{n \in \Z} \cone_{2n+1}(h_*)
  \to \oplus_{n \in \Z} \cone_{2n}(h_*)
  \]
  is an isomorphism because $(c_* + \delta_*)_{\odd} \circ (c_* +
  \delta_*)_{\even}$ is given by a a lower triangle matrix with $1$ on
  the diagonal (compare~\cite{Luc2}*{Lemma~3.40 on p.~145}).  By
  Lemma~\ref{lem:MDC for projective modules}, for every homomorphism
  of finitely generated based projective $\chi_B
  L^\infty(Y,\Z)*H\chi_B$-modules, the determinant of the induced
  morphisms of Hilbert ${\mathcal N}(H \curvearrowright Y|_B)$-modules
  is $\ge 1$.  Since the determinant is multiplicative for
  isomorphisms~\cite{Luc2}*{Theorem~3.14 on p.~128}, we conclude that
  any isomorphism of Hilbert ${\mathcal N}(H \curvearrowright
  Y|_B)$-modules which comes from a $\chi_B
  L^\infty(Y,\Z)*H\chi_B$-isomorphism has determinant $1$. In
  particular, we obtain $\det_{{\mathcal N}(H \curvearrowright
    Y|_B)}((c^{(2)}_* + \delta^{(2)}_*)_{\odd}) = 1$ and
  $\rho^{(2)}(\cone_*(h^{(2)}_*)) = 0$.
\end{proof}

\section{A cautionary example}\label{sec:caveat}

In Gaboriau's proof~\cite{gaboriau-main} of the orbit equivalence
invariance of $L^2$-Betti numbers one encounters the following
situation: One obtains a homotopy equivalence between Hilbert
$\caln(\calr)$-complexes $C_\ast$ and $D_\ast$ which is not bounded --
unless the given orbit equivalence is uniform. To obtain an estimate
(and by symmetry an equality) between the $L^2$-Betti numbers of these
complexes, one constructs increasing sequences of subcomplexes
$C_\ast^{(k)}\subset C_\ast$ and $ D_\ast^{(k)}\subset D_\ast$ such
that for every $n\ge 0$ the closures of $\bigcup_k C_n^{(k)}$ and
$\bigcup_k D_n^{(k)}$ are $C_n$ and $D_n$, respectively, and bounded
homotopy retracts $C_\ast^{(k)}\rightarrow D_\ast^{(k)}$ for every
$k\in\bbN$. The needed continuity property for the $L^2$-Betti numbers
boils down to the following (easy) continuity property of the von
Neumann trace:

Let $\cala$ be a finite von Neumann algebra and $M$ a finitely
generated Hilbert $\cala$-module. Let $f:M\rightarrow M$ be a positive
$\cala$-morphism and $p_k:M\rightarrow M$ a sequence of
$\cala$-equivariant projections that weakly converge to the
identity. Then
\[
\tr_\cala(f)=\lim_{k\rightarrow\infty}\tr_\cala(f\circ p_k).
\]

To be able to drop the uniformity assumption in
Theorem~\ref{thm:invariance under uniform ME} one would want, among
other things, a similar continuity property of the Fuglede-Kadison
determinant. The following theorem, whose proof we omit, states such
-- but with an important, restrictive assumption on the kernels:

\begin{theorem}\label{the: Determinants and exhausting projections}
  Let $f \colon U\rightarrow V$ be a morphism of finitely generated
  Hilbert $\cala$-modules.  Let $p_k \colon U \to U$ for $k\in\bbN$ be
  a sequence of projections with $\ker(f) \subseteq \im(p_k) \subseteq
  \im(p_{k+1})$ that weakly converges to the identity. Then
  \[\determ_\cala(f)=\lim_{k \to \infty} \determ_\cala(f \circ p_k).\]
\end{theorem}

Next we will show that the condition $\ker(f)\subset\im(p_k)$ is
necessary.  In our example, $\cala$ is $L^{\infty}([0,1])$, and $f$ is
the projection onto the second factor $\pr_2 \colon L^2([0,1]) \oplus
L^2([0,1]) \to L^2([0,1])$. We drop the subscript $L^\infty([0,1])$ in
the notation of the trace and the determinant.  Obviously,
$\det(\pr_2) = 1$. Let $\epsilon_k$ for $k = 1,2, \ldots$ be any
sequence of positive real numbers.  Consider the morphism of finitely
generated Hilbert $L^{\infty}([0,1])$-modules
\begin{gather*}
  u_k \colon L^2([1-2^{1-k},1-2^{-k}])\rightarrow L^2([0,1]) \oplus L^2([0,1])\\
  \phi\mapsto \Bigl(\frac{\epsilon_k}{\sqrt{1 + \epsilon_k^2}} \cdot
  \overline{\phi}, \frac{1}{\sqrt{1 + \epsilon_k^2}} \cdot
  \overline{\phi}\Bigr),
\end{gather*}
where $\overline{\phi} \in L^2([0,1])$ is obtained from $\phi$ by
extending by zero on the complement of
$[1-2^{1-k},1-2^{-k}]$. Consider the morphism of finitely generated
Hilbert $L^{\infty}([0,1])$-modules
\begin{gather*}
  v_k \colon L^2([0,1-2^{1-k}]) \oplus L^2([0,1-2^{1-k}])\rightarrow
  L^2([0,1]) \oplus L^2([0,1])\\
  (\phi_1,\phi_2)\mapsto
  \bigl(\overline{\phi_1},\overline{\phi_2}\bigr).
\end{gather*}
The morphisms $u_k$ and $v_k$ are isometric
$L^{\infty}([0,1])$-embeddings, and $\im(u_k)$ and $\im(v_k)$ are
othogonal to one another. Then
\[A_k ~ = ~ \im(u_k) \oplus \im(v_k).\] is an increasing sequence of
closed subspaces.  Since
\[\dim(A_k) = 2 \cdot (1-2^{1-k}) + 2^{1-k} -
2^{-k}\xrightarrow{k\rightarrow\infty} 2,\] the sequence of
projections defined by
\[p_k \colon L^2([0,1]) \oplus L^2([0,1]) \to L^2([0,1]) \oplus
L^2([0,1]),~\Im(p_k)=A_k\] weakly converges to the
identity. Furthermore, we have
\begin{align*}
  \det(\pr_2 \circ p_k) =\det(\pr_2|_{A_k})
  &= \det\bigl(\pr_2 \circ (u_k \oplus v_k)\bigr) \\
  &=\det\bigl((\pr_2 \circ u_k)  \oplus (\pr_2 \circ v_k)\bigr)\\
  &=\det(\pr_2 \circ u_k)  \cdot \det(\pr_2 \circ v_k)\\
  &=\det\Bigl(\frac{1}{\sqrt{1 + \epsilon_k^2}} \cdot
  \id_{L^2([1-2^{1-k},1-2^{-k}])}\Bigr) \\
  & \hspace{10mm} \cdot \det\bigl(\pr_2|_{L^2([0,1-2^{1-k}]) \oplus
    L^2([0,1-2^{1-k}])}\bigr)
  \\
  &=\Bigl(\frac{1}{\sqrt{1 +
      \epsilon_k^2}}\Bigr)^{\dim(L^2([1-2^{1-k},1-2^{-k}]))}\cdot 1\\
  &=
  \Bigl(\frac{1}{\sqrt{1 + \epsilon_k^2}}\Bigr)^{2^{1-k} - 2^{-k}}\\
  &=\bigl(1 + \epsilon_k^2\bigr)^{-2^{-k-1}}.
\end{align*}
If we choose $\epsilon_k= \sqrt{k^{2^{k+1}}-1}$, then we obtain that
\begin{align*}
  \det(\pr_2 \circ p_k)  & = \frac{1}{k}\text{ for $k =1,2 \dots$;} \\
  \det(\pr_2) &=1.
\end{align*}


\begin{bibdiv}
\begin{biblist}
	
	\bib{burgh}{article}{
	   author={Burghelea, D.},
	   author={Friedlander, L.},
	   author={Kappeler, T.},
	   author={McDonald, P.},
	   title={Analytic and Reidemeister torsion for representations in finite
	   type Hilbert modules},
	   journal={Geom. Funct. Anal.},
	   volume={6},
	   date={1996},
	   number={5},
	   pages={751--859},
	}
	\bib{connes}{article}{
	   author={Connes, A.},
	   title={Classification of injective factors. Cases $II\sb{1},$
	   $II\sb{\infty },$ $III\sb{\lambda },$ $\lambda \not=1$},
	   journal={Ann. of Math. (2)},
	   volume={104},
	   date={1976},
	   number={1},
	   pages={73--115},
	}
	
	\bib{dixmier}{book}{
	   author={Dixmier, Jacques},
	   title={von Neumann algebras},
	   series={North-Holland Mathematical Library},
	   volume={27},
	   publisher={North-Holland Publishing Co.},
	   place={Amsterdam},
	   date={1981},
	   pages={xxxviii+437},
	}
	
	\bib{approx}{article}{
	   author={Dodziuk, J{\'o}zef},
	   author={Linnell, Peter},
	   author={Mathai, Varghese},
	   author={Schick, Thomas},
	   author={Yates, Stuart},
	   title={Approximating $L\sp 2$-invariants and the Atiyah conjecture},
	   note={Dedicated to the memory of J\"urgen K. Moser},
	   journal={Comm. Pure Appl. Math.},
	   volume={56},
	   date={2003},
	   number={7},
	   pages={839--873},
	}

	\bib{dye1}{article}{
	   author={Dye, H. A.},
	   title={On groups of measure preserving transformation. I},
	   journal={Amer. J. Math.},
	   volume={81},
	   date={1959},
	   pages={119--159},
	}

	\bib{dye2}{article}{
	   author={Dye, H. A.},
	   title={On groups of measure preserving transformations. II},
	   journal={Amer. J. Math.},
	   volume={85},
	   date={1963},
	   pages={551--576},
	}

	\bib{ES1}{article}{
	   author={Elek, G{\'a}bor},
	   author={Szab{\'o}, Endre},
	   title={Hyperlinearity, essentially free actions and $L\sp 2$-invariants.
	   The sofic property},
	   journal={Math. Ann.},
	   volume={332},
	   date={2005},
	   number={2},
	   pages={421--441},
	}

	\bib{feldman+mooreI}{article}{
	   author={Feldman, Jacob},
	   author={Moore, Calvin C.},
	   title={Ergodic equivalence relations, cohomology, and von Neumann
	   algebras. I},
	   journal={Trans. Amer. Math. Soc.},
	   volume={234},
	   date={1977},
	   number={2},
	   pages={289--324},
	}

	\bib{feldman+mooreII}{article}{
	   author={Feldman, Jacob},
	   author={Moore, Calvin C.},
	   title={Ergodic equivalence relations, cohomology, and von Neumann
	   algebras. II},
	   journal={Trans. Amer. Math. Soc.},
	   volume={234},
	   date={1977},
	   number={2},
	   pages={325--359},
	}
	
	\bib{Fur1}{article}{
	   author={Furman, Alex},
	   title={Gromov's measure equivalence and rigidity of higher rank lattices},
	   journal={Ann. of Math. (2)},
	   volume={150},
	   date={1999},
	   number={3},
	   pages={1059--1081},
	}

	\bib{Fur2}{article}{
	   author={Furman, Alex},
	   title={Orbit equivalence rigidity},
	   journal={Ann. of Math. (2)},
	   volume={150},
	   date={1999},
	   number={3},
	   pages={1083--1108},
	}

	\bib{Fur3}{article}{
	   author={Furman, Alex},
	   title={Outer automorphism groups of some ergodic equivalence relations},
	   journal={Comment. Math. Helv.},
	   volume={80},
	   date={2005},
	   number={1},
	   pages={157--196},
	}



	\bib{gaboriau-main}{article}{
	   author={Gaboriau, Damien},
	   title={Invariants $l\sp 2$ de relations d'\'equivalence et de groupes},
	   language={French},
	   journal={Publ. Math. Inst. Hautes \'Etudes Sci.},
	   number={95},
	   date={2002},
	   pages={93--150},
	}
	
	\bib{gaboriau-examples}{article}{
	   author={Gaboriau, Damien},
	   title={Examples of groups that are measure equivalent to the free group},
	   journal={Ergodic Theory Dynam. Systems},
	   volume={25},
	   date={2005},
	   number={6},
	   pages={1809--1827},
	}
	

	\bib{Gro}{article}{
	   author={Gromov, Mikhail},
	   title={Asymptotic invariants of infinite groups},
	   conference={
	      title={Geometric group theory, Vol.\ 2},
	      address={Sussex},
	      date={1991},
	   },
	   book={
	      series={London Math. Soc. Lecture Note Ser.},
	      volume={182},
	      publisher={Cambridge Univ. Press},
	      place={Cambridge},
	   },
	   date={1993},
	   pages={1--295},
	}

	\bib{jones}{book}{
	   author={Jones, Vaughan},
	   author={Sunder, V. S.},
	   title={Introduction to subfactors},
	   series={London Mathematical Society Lecture Note Series},
	   volume={234},
	   publisher={Cambridge University Press},
	   date={1997},
	}

	\bib{lam}{book}{
	   author={Lam, T. Y.},
	   title={Lectures on modules and rings},
	   series={Graduate Texts in Mathematics},
	   volume={189},
	   publisher={Springer-Verlag},
	   place={New York},
	   date={1999},
	   pages={xxiv+557},
	}
	\bib{lueck-gafa}{article}{
	   author={L{\"u}ck, W.},
	   title={Approximating $L\sp 2$-invariants by their finite-dimensional
	   analogues},
	   journal={Geom. Funct. Anal.},
	   volume={4},
	   date={1994},
	   number={4},
	   pages={455--481},
	}

	\bib{Luc1}{article}{
	   author={L{\"u}ck, Wolfgang},
	   title={Hilbert modules and modules over finite von Neumann algebras and
	   applications to $L\sp 2$-invariants},
	   journal={Math. Ann.},
	   volume={309},
	   date={1997},
	   number={2},
	   pages={247--285},
	}

	\bib{Luc2}{book}{
	   author={L{\"u}ck, Wolfgang},
	   title={$L\sp 2$-invariants: theory and applications to geometry and
	   $K$-theory},
	   series={Ergebnisse der Mathematik und ihrer Grenzgebiete. 3. Folge. A
	   Series of Modern Surveys in Mathematics [Results in Mathematics and
	   Related Areas. 3rd Series. A Series of Modern Surveys in Mathematics]},
	   volume={44},
	   publisher={Springer-Verlag},
	   place={Berlin},
	   date={2002},
	   pages={xvi+595},
	}
	
	\bib{popa-survey}{article}{
	   author={Popa, Sorin},
	   title={Deformation and rigidity for group actions and von Neumann
	   algebras},
	   conference={
	      title={International Congress of Mathematicians. Vol. I},
	   },
	   book={
	      publisher={Eur. Math. Soc., Z\"urich},
	   },
	   date={2007},
	   pages={445--477},
	}
	
	\bib{sauer-promotion}{thesis}{
		author={Sauer, Roman}, 
		title={$L^2$-Invariants of groups and discrete measured groupoids}, 
		type={Dissertation}, 
		organization={WWU M\"unster}, 
		eprint={http://nbn-resolving.de/urn:nbn:de:hbz:6-85659549583}, 
	}
	
	\bib{Sau}{article}{
	   author={Sauer, Roman},
	   title={$L\sp 2$-Betti numbers of discrete measured groupoids},
	   journal={Internat. J. Algebra Comput.},
	   volume={15},
	   date={2005},
	   number={5-6},
	   pages={1169--1188},
	}

	\bib{sauer-amenable}{article}{
	   author={Sauer, Roman},
	   title={Homological invariants and quasi-isometry},
	   journal={Geom. Funct. Anal.},
	   volume={16},
	   date={2006},
	   number={2},
	   pages={476--515},
	}

	\bib{Sch}{article}{
	   author={Schick, Thomas},
	   title={$L\sp 2$-determinant class and approximation of $L\sp 2$-Betti
	   numbers},
	   journal={Trans. Amer. Math. Soc.},
	   volume={353},
	   date={2001},
	   number={8},
	   pages={3247--3265 (electronic)},
	}

	\bib{shalom-amenable}{article}{
	   author={Shalom, Yehuda},
	   title={Harmonic analysis, cohomology, and the large-scale geometry of
	   amenable groups},
	   journal={Acta Math.},
	   volume={192},
	   date={2004},
	   number={2},
	   pages={119--185},
	}
	\bib{take}{book}{
	   author={Takesaki, M.},
	   title={Theory of operator algebras. III},
	   series={Encyclopaedia of Mathematical Sciences},
	   volume={127},
	   note={Operator Algebras and Non-commutative Geometry, 8},
	   publisher={Springer-Verlag},
	   place={Berlin},
	   date={2003},
	   pages={xxii+548},
	}

	\bib{wegner}{article}{
	   author={Wegner, Christian},
	   title={$L\sp 2$-invariants of finite aspherical CW-complexes},
	   note={to appear in manuscripta mathematica},
	}

\end{biblist}
\end{bibdiv}
\end{document}